\documentclass[11pt]{article}

\usepackage{enumerate}
\usepackage{amsmath}
\usepackage{xspace}
\usepackage{amssymb}
\usepackage{stmaryrd}
\usepackage[dvips]{graphics}
\usepackage{proof}
\usepackage{latexsym}  
\usepackage{euscript}  
\usepackage{bm}

\input xy
\xyoption{all}
\xyoption{2cell}
\UseAllTwocells
\newtheorem{observation}{Remark}[section]
\newtheorem{lemma}[observation]{Lemma}  
\newtheorem{theorem}[observation]{Theorem}
\newtheorem{definition}[observation]{Definition}
\newtheorem{example}[observation]{Example}
\newtheorem{remark}[observation]{Remark}

\newtheorem{proposition}[observation]{Proposition}
\newtheorem{corollary}[observation]{Corollary}

\newcommand{\linear}{{linear}\xspace}

\newcommand{\proof}{\noindent{\sc Proof:}\xspace}
\def\endproof{~\hfill$\Box$\vskip 10pt}

\newcommand{\ox}{\otimes}
\def\*{\otimes}
\newcommand{\x}{\times}
\newcommand{\<}{\langle}
\renewcommand{\>}{\rangle}
\mathcode`\:="603A    
\mathchardef\colon="303A 
\mathcode`\<="4268    
\mathcode`\>="5269    
\mathcode`\|="326A    
\mathchardef\gt="313E 
\mathchardef\lt="313C 
\newcommand{\scr}{\scriptsize}

\newcommand{\lollipop}{\multimap}

\renewcommand{\L}{{\ensuremath{\cal L}\xspace}}

\renewcommand{\S}{{\ensuremath{\cal S}\xspace}}

\newcommand{\X}{\ensuremath{\mathbb X}\xspace}
\newcommand{\Y}{\ensuremath{\mathbb Y}\xspace}

\newcommand{\ex}{{\rm ex}\xspace}
\newcommand{\eval}{{\rm eval}\xspace}
\newcommand{\curry}{{\rm curry}\xspace}
\newcommand{\diff}[4]{\frac{\partial #1}{\partial #2}\left( #3 \right)\cdot #4}

\newcommand\nats{\hbox{$I \kern - .38em N$}} 
\newcommand\ints{\hbox{$Z \kern - .65em Z$}} 

\makeatletter

\newdimen\w@dth

\def\setw@dth#1#2{\setbox\z@\hbox{\scriptsize $#1$}\w@dth=\wd\z@
\setbox\@ne\hbox{\scriptsize $#2$}\ifnum\w@dth<\wd\@ne \w@dth=\wd\@ne \fi
\advance\w@dth by 1.2em}

\def\t@^#1_#2{\allowbreak\def\n@one{#1}\def\n@two{#2}\mathrel
{\setw@dth{#1}{#2}
\mathop{\hbox to \w@dth{\rightarrowfill}}\limits
\ifx\n@one\empty\else ^{\box\z@}\fi
\ifx\n@two\empty\else _{\box\@ne}\fi}}
\def\t@@^#1{\@ifnextchar_ {\t@^{#1}}{\t@^{#1}_{}}}

\def\t@left^#1_#2{\def\n@one{#1}\def\n@two{#2}\mathrel{\setw@dth{#1}{#2}
\mathop{\hbox to \w@dth{\leftarrowfill}}\limits
\ifx\n@one\empty\else ^{\box\z@}\fi
\ifx\n@two\empty\else _{\box\@ne}\fi}}
\def\t@@left^#1{\@ifnextchar_ {\t@left^{#1}}{\t@left^{#1}_{}}}

\def\two@^#1_#2{\def\n@one{#1}\def\n@two{#2}\mathrel{\setw@dth{#1}{#2}
\mathop{\vcenter{\hbox to \w@dth{\rightarrowfill}\kern-1.7ex
                 \hbox to \w@dth{\rightarrowfill}}%
       }\limits
\ifx\n@one\empty\else ^{\box\z@}\fi
\ifx\n@two\empty\else _{\box\@ne}\fi}}
\def\tw@@^#1{\@ifnextchar_ {\two@^{#1}}{\two@^{#1}_{}}}

\def\tofr@^#1_#2{\def\n@one{#1}\def\n@two{#2}\mathrel{\setw@dth{#1}{#2}
\mathop{\vcenter{\hbox to \w@dth{\rightarrowfill}\kern-1.7ex
                 \hbox to \w@dth{\leftarrowfill}}%
       }\limits
\ifx\n@one\empty\else ^{\box\z@}\fi
\ifx\n@two\empty\else _{\box\@ne}\fi}}
\def\t@fr@^#1{\@ifnextchar_ {\tofr@^{#1}}{\tofr@^{#1}_{}}}

\newdimen\W@dth
\def\setW@dth#1#2{\setbox\z@\hbox{$#1$}\W@dth=\wd\z@
\setbox\@ne\hbox{$#2$}\ifnum\W@dth<\wd\@ne \W@dth=\wd\@ne \fi
\advance\W@dth by 1.2em}

\def\T@^#1_#2{\allowbreak\def\N@one{#1}\def\N@two{#2}\mathrel
{\setW@dth{#1}{#2}
\mathop{\hbox to \W@dth{\rightarrowfill}}\limits
\ifx\N@one\empty\else ^{\box\z@}\fi
\ifx\N@two\empty\else _{\box\@ne}\fi}}
\def\T@@^#1{\@ifnextchar_ {\T@^{#1}}{\T@^{#1}_{}}}

\def\T@left^#1_#2{\def\N@one{#1}\def\N@two{#2}\mathrel{\setW@dth{#1}{#2}
\mathop{\hbox to \W@dth{\leftarrowfill}}\limits
\ifx\N@one\empty\else ^{\box\z@}\fi
\ifx\N@two\empty\else _{\box\@ne}\fi}}
\def\T@@left^#1{\@ifnextchar_ {\T@left^{#1}}{\T@left^{#1}_{}}}

\def\Tofr@^#1_#2{\def\N@one{#1}\def\N@two{#2}\mathrel{\setW@dth{#1}{#2}
\mathop{\vcenter{\hbox to \W@dth{\rightarrowfill}\kern-1.7ex
                 \hbox to \W@dth{\leftarrowfill}}%
       }\limits
\ifx\N@one\empty\else ^{\box\z@}\fi
\ifx\N@two\empty\else _{\box\@ne}\fi}}
\def\T@fr@^#1{\@ifnextchar_ {\Tofr@^{#1}}{\Tofr@^{#1}_{}}}

\def\Two@^#1_#2{\def\N@one{#1}\def\N@two{#2}\mathrel{\setW@dth{#1}{#2}
\mathop{\vcenter{\hbox to \W@dth{\rightarrowfill}\kern-1.7ex
                 \hbox to \W@dth{\rightarrowfill}}%
       }\limits
\ifx\N@one\empty\else ^{\box\z@}\fi
\ifx\N@two\empty\else _{\box\@ne}\fi}}
\def\Tw@@^#1{\@ifnextchar_ {\Two@^{#1}}{\Two@^{#1}_{}}}

\def\to{\@ifnextchar^ {\t@@}{\t@@^{}}}
\def\from{\@ifnextchar^ {\t@@left}{\t@@left^{}}}
\def\tofro{\@ifnextchar^ {\t@fr@}{\t@fr@^{}}}
\def\To{\@ifnextchar^ {\T@@}{\T@@^{}}}
\def\From{\@ifnextchar^ {\T@@left}{\T@@left^{}}}
\def\Two{\@ifnextchar^ {\Tw@@}{\Tw@@^{}}}
\def\Tofro{\@ifnextchar^ {\T@fr@}{\T@fr@^{}}}

\makeatother
\newcommand{\avoidwidow}{\par\pagebreak[3]\kern50pt \pagebreak[3]\kern-50pt}

\title{Cartesian differential storage categories}
\author{R. Blute \and J.R.B. Cockett \and R.A.G. Seely}

\begin{document}


\maketitle

\begin{abstract}
Cartesian differential categories were introduced to provide an abstract
axiomatization of categories of differentiable functions. The fundamental example
is the category whose objects are Euclidean spaces and whose arrows are smooth maps.

Tensor differential categories provide the framework for categorical models
of differential linear logic. The coKleisli category of any tensor differential category
is always a Cartesian differential category.  Cartesian differential categories, besides arising
in this manner as coKleisli categories,  occur in many different
and quite independent ways.  Thus, it was not obvious how to pass from Cartesian
differential categories back to tensor differential categories.

This paper provides natural conditions under which the linear maps of a Cartesian differential
category form a tensor differential category. This is a question of some practical importance
as much of the machinery of modern differential geometry is based on models which
implicitly allow such a passage, and thus the results and tools of the area tend to freely
assume access to this structure.

The purpose of this paper is to make precise the connection between the
two types of differential categories.  As a prelude to this, however, it
is convenient to have available a general theory which relates the behaviour of
``linear'' maps in Cartesian categories to the structure of Seely categories.  The
latter were developed to provide the categorical semantics for (fragments of)
linear logic which use a ``storage'' modality.    The general theory of storage, which
underlies the results mentioned above, is developed in the opening
sections of the paper and is then applied to the case of differential categories.
\end{abstract}

\section{Introduction}

A fundamental observation of Girard, \cite{Girard}, which led to the
development of linear logic, was that the hom-functor of stable domains
could be decomposed as $A \Rightarrow B \colon=~ !A \lollipop B$ where
the latter, $X \lollipop Y$, is the hom-functor of coherence spaces.
This suggested that one might similarly be able to decompose
differentiable maps between two spaces $A$ and $B$ as the ``linear''
maps from a constructed space $S(A)$ to $B$.  Ehrhard and Regnier's work
\cite{ER04} on the differential $\lambda$-calculus was inspired by this
analogy and in \cite{E04} Ehrhard provided concrete examples which
realized this analogy.  This raised the question of how general a
correspondence this was.

In \cite{diffl} the notion of a (tensor) differential category was
introduced as a minimal categorical doctrine in which differentiability
could be studied.  This generalized Ehrhard and Regnier's  ideas in various ways.
It dispensed with the necessity that the setting be $*$-autonomous: the
decomposition of the hom-functor was then handled through the presence
of a comonad---or modality---satisfying a minimal set  of coherences.  Finally
the differential was introduced satisfying a minimal set of identities.  In this paper
we shall, in fact, consider a slight strengthening of this basic notion by requiring,
first, that the modality be a storage modality ({\em i.e.} the Seely isomorphisms
$1 \equiv S(0)$ and $S(A) \ox S(B) \equiv S(A \x B)$ are present) and, second,
that the differential satisfies an interchange law.

This maintained the perspective of linear logic, by providing a
decomposition of differentiable functions through a comonad---or
storage modality---on more basic ``linear'' functions.  However, a
rather unsatisfactory aspect of this perspective remained.  Smooth
functions ({\em i.e.} infinitely differentiable functions)---after all the
main subject of calculus---appeared only indirectly
as the maps of the coKleisli category. This left a veil between these settings and the
direct understanding of these differentiable functions in the classical sense.   It seemed,
therefore, important to develop a more direct view of what a category of
differentiable functions should look like.  This caused us to examine more
closely the coKleisli categories of differential categories and to seek
a direct axiomatization for them.

To advance this aim we deployed (the dual of) Fuhrmann's notion of an
abstract Kleisli category \cite{fuhrmann}.  In the course of developing these ideas we
stumbled on a much deeper sense in which the situation in linear logic
can be read as mathematics.  Thus a key result of this paper is a very
general structural theorem about how ``models of linear logic'' arise.
Rather than retrospectively remove our path to this result, we have
centred the paper around the results on differential categories so that
the motivation which brought us to this particular passage is not lost.

The ``models of linear logic'' we consider here are not, in fact, models
of full linear logic, as the underlying categories are monoidal and not,
in general,  monoidal closed, let alone $*$-autonomous. These models were
studied in depth by Gavin Bierman \cite{Bierman}: he called them ``Seely
categories''.  Given that these categories have now been the object of
quite a number of studies, it is perhaps surprising that there is more to say about
them.  Our main rather general structural theorem points out that these
categories arise in a rather natural mathematical way from the ``linear
maps'' of Cartesian storage categories.  As Cartesian storage categories
have a rather natural mathematical genus, this helps to explain why
``Seely categories'' are such fundamental structures.

Our examination of the coKleisli categories of differential categories
had led us to introduce the notion of a Cartesian differential category,
\cite{CartDiff}.  Cartesian differential categories are categories of
``raw'' differentiable maps with all the indirection---alluded to
above---removed.  Pleasingly these categories have a wide range of
models.  In particular,  in \cite{CartDiff}, we showed that the
coKleisli category of a (tensor) differential category---satisfying an
additional interchange axiom---{\em is\/} a Cartesian differential
category.

However, this had left open the issue of providing a converse for this
result.  In general, it is certainly {\em not\/} the case that a
Cartesian differential category is the coKleisli category of a (tensor)
differential category---a basic counterexample is the category of finite dimensional
real vector spaces with smooth maps.   Thus, it is natural to ask what extra
conditions are required on a Cartesian differential category to make it
a coKleisli category of a (tensor) differential category. This
paper provides a rather natural (and appealing) answer to this
question.  However, lest the reader think we have completely
resolved this question, we hasten to admit that it still does not answer the question in
complete generality.  Furthermore, it leaves some other rather natural
questions open: for example it is natural to wonder whether {\em every\/}
Cartesian differential category arises as a full subcategory of the
coKleisli category of a (tensor) differential category.

The motivation for developing Cartesian storage categories was that they
provided an intermediate description of the coKleisli category of a
(tensor) differential category.  These coKleisli categories were already
proven  to be Cartesian differential categories and, as such, had an
obvious subcategory of linear maps---namely those which were linear in
the natural differential sense.  A Cartesian storage category uses the abstract
properties of ``linear maps'' as the basis for its definition and is always the coKleisli category
of its linear maps.  In order for such a storage category to be the coKleisli category of
a ``Seely category'', a further ingredient was required: one must be able to represent
bilinear maps with a tensor product.  The theory of these categories---which we have taken
to thinking of as the theory of storage---provides the theoretical core of the paper
and is described in Sections \ref{storage-categories} and \ref{tensor-storage-cats}.

In Section \ref{diff-storage} we apply this theory to Cartesian differential categories.
Cartesian differential categories exhibit a notable further coincidence of
structure.  When one can classify the linear maps of a Cartesian
differential category one also automatically has a transformation called a ``codereliction''.
Having a codereliction in any Cartesian storage category
implies---assuming linear idempotents split---that one also has, for free,
tensorial representation.  Thus, to apply the theory of storage categories to
Cartesian differential categories it suffices to know just that the linear maps are classified.
Finally, but crucially for the coherence of the story, the linear maps then, with no further
structural requirement, form a tensor differential category.  This provides the main
result of the paper Theorem \ref{lin-of-cdsc}.  Completing the circle then,
at the end of the paper, we are in a position to point out that a closed Cartesian
differential storage category is a model of Ehrhard and Regnier's
differential $\lambda$-calculus \cite{ER04}.

Considering that we have been discussing two types of differential
category, it seemed a good idea to qualify each with an appropriate
adjective, so as to distinguish between them in a balanced manner.  So
we shall (in this paper) refer to differential categories as ``tensor
differential categories'', to contrast them with ``Cartesian
differential categories''.    In parallel with this, we introduce other
terminological pairings which are natural from the perspective of this
paper, but are also structures which have received attention under
different names: For example,  ``tensor storage categories'' are
essentially ``Seely categories'' and ``tensor differential storage
categories''  are essentially what previously we called differential
storage categories.    Our purpose is not to precipitate any renaming of
familiar structures but rather to emphasize the structural relationships
which emerge from the story we tell here.

\section{Cartesian Storage categories}
\label{storage-categories}

In the discussion above two situations were described in which we selected from
a Cartesian category certain maps, namely the coherent maps in stable domains
and the linear maps among differentiable maps, and then used a classification of these maps
to extract a Seely category.   The aim of the next two sections is to show
why this works as a general procedure and, moreover, can
produce a tensor storage (or ``Seely'') category . In fact, our aim is to prove
much more: namely that {\em all\/}  ``Seely categories'' (with an exact modality)
arise in this manner.

We start by defining what is meant by a system \L\ of \linear maps. One way
to view a \linear system is as a subfibration of the simple fibration.  In
a category equipped with a system of \linear maps, we shall often
suppress explicit mention of $\cal L$, referring to maps belonging to
this system as simply being {\em linear\/}.  Such systems arise very
naturally in Cartesian closed categories from exponentiable systems of
maps and we make a detour to describe these connections.

Our initial aim is to characterize the categories in which there is
a decomposition of maps analogous to that in linear logic.   This
eventually leads us to ``Seely categories'' as those in which all the
``first-order'' structural components of the logic are present.

Leaving definitions for later, the path will be the following:
a ``Cartesian storage category'' is a category equipped with a system of
\linear maps which are ``strongly and persistently classified''.  We then
prove that this is equivalent to demanding that the category is a ``strong
abstract coKleisli category''.  This, in turn, is equivalent
to being the coKleisli category of a category with a ``forceful''
comonad.  Finally, to this story we add, in the next section,   ``tensorial representation'', which
when present, ensures the linear maps forms a tensor storage
category---this is essentially a ``Seely category''.


\subsection{Systems of \linear\ maps}\label{linear-maps}

By a {\bf system of maps} of a category is meant a subcategory on the same objects: a
{\bf Cartesian} system of maps of a Cartesian category is just such a
subcategory closed under the product structure.  This means that
projections, diagonal maps, and final map must be in the system and,
furthermore, the pairing of any two maps in the system must be in the
system.  For a system \L\ of maps to be \linear it must determine a
subfibration of the simple fibration.

The simple slice $\X[A]$ at $A$---that is the fiber over $A$ of the
simple fibration---of a Cartesian category $\X$, has the same objects
as $\X$ but the homsets are modified $\X[A](X,Y) = \X(A \x X,Y)$. The
composition in $\X[A]$ of two maps $g: A \x X \to Y$ and $h: A \x Y \to
Z$ is $\< \pi_0,g \> h$ with identities given by projection $\pi_1: A \x
X \to X$.  The substitution functor between simple slices has the
following effect on the maps of $\X$.  If $g: A \to B$ is any map of
$\X$, a map of $\X[B]$ is, in $\X$, a map $g: B \x X \to Y$, then
$\X[f](g)= (f \x 1)g: A \x X \to Y$.

\begin{definition}
A Cartesian category $\X$ is said to have a {\bf system of \linear
maps}, ${\cal L}$, in case in each simple slice $\X[A]$ there is a system
of maps ${\cal L}[A] \subseteq \X[A]$, which we shall refer to as the
{\bf ${\cal L}[A]$-linear} maps (or, when $\cal L$ is understood, simply
as ``linear'' maps), satisfying:
\begin{enumerate}[{\bf [LS.1]}]
\item  for each $X,Y,Z \in \X[A]$, all of $1_X,\pi_0: X \x Y \to
X,\pi_1: X \x Y \to Y$ are in ${\cal L}[A]$ and, furthermore, $\< f,g\>:
X \to Y \x Z \in {\cal L}[A]$ whenever $f,g \in {\cal L}[A]$;
\item  in each $\X[A]$, ${\cal L}[A]$, as a system of maps,  is closed under composition, furthermore,
whenever $g \in {\cal L}[A]$ is a retraction and $gh \in {\cal L}[A]$
then $h \in {\cal L}[A]$;
\item  all substitution functors
$$\infer{\X[B] \to_{\X[f]} \X[A]}{A \to^f B}$$
preserve linear maps.
\end{enumerate}
\end{definition}

Notice how a system of \linear maps determines a subfibration:
$$\xymatrix{{\cal L}[\X] \ar[d]^{\partial_{\cal L}} \ar[rr]^{{\cal I}[\X]}
  & & \X[\X] \ar[d]^{\partial} \\
  {\cal L} \ar[rr]_{{\cal I}} & & \X}$$
where ${\cal  I}$ is the obvious inclusion and ${\cal L}[\X]$ is the category:
\begin{description}
\item{{\bf [Objects:]}} $(A,X) \in \X_0 \x \X_0$
\item{{\bf [Maps:]}} $(f,g): (A,X) \to (B,Y)$ where $g: A \to B$ is any
map and $f: A \x X \to Y$ is linear in
its second argument.
\end{description}
Composition and identities are as in $\X[\X]$, where the maps are the
same but with no linear restriction.  Composition is defined by $(f,g)
(f',g') = (\< \pi_0 g,f \> f',g g')$ and $(\pi_1,1)$ is the identity
map.  Cartesian maps are of the form $(\pi_1,g): (A,X) \to (B,X)$.

Furthermore we observe using {\bf [LS.2]}:

\begin{lemma}
For any $X$ with a system of \linear maps in any slice $\X[A]$:
\begin{enumerate}[(i)]
\item  if $f$ is an linear map  which is an isomorphism, then $f^{-1}$ is linear;
\item  if $e$ is a linear idempotent which splits as a retraction $r$
and a section $i$, and if $r$ is linear, then $i$ is also linear.
\end{enumerate}
\end{lemma}

We shall need terminology which focuses on the linear argument rather
than the context of that argument.  Thus to say that a map $f: X \x A
\to Y$ is linear {\em in its second argument\/} is equivalent to saying
that $f$ is ${\cal L}[X]$-linear. Similarly, a map, $f$,  is linear {\em in its first argument\/} if
$c_{\x} f$ is ${\cal L}[X]$-linear.

\begin{lemma}
If ${\cal L}$ is a system of \linear maps on a Cartesian category $\X$
then if $f: X \x A \to B$ and $g: Y \x B \to C$ are linear in their
second arguments then $(f \x 1)g: Y \x X \x A \to C$ is linear in its
third argument.
\end{lemma}

\proof
As $f$ is linear in $\X[X]$, $\X[\pi_1](f): Y \x X \x A \to C$ is linear
in $\X[Y \x X]$. Similarly as $g$ is linear in $\X[Y]$, $\X[\pi_0](g)$ is
linear in $\X[Y \x X]$.  Now $(1 \x f)g = \X[\pi_1](f)\X[\pi_0](g)$
which makes it a composite of linear maps and so itself linear.
\endproof

The lemma means that to determine whether a composite is linear in a
particular argument it suffices to know that each composition was at a
linear argument.

If ${\cal L}$ is a system of \linear maps then $f: A \x B \to C$ is
(${\cal L}$-){\bf bilinear} in case it is linear in each argument
individually.

Finally we make the useful if immediate observation:

\begin{lemma}
If $\X$ has a system of linear maps ${\cal L}$ then each simple slice, $\X[A]$, inherits a system of linear maps ${\cal L}[A]$.
\end{lemma}


\subsection{Closed systems of maps}
\label{closed-systems}

A system of \linear maps induces a system of maps on $\X$ itself, as $\X
\cong \X[1]$. This class, which we shall denote by ${\cal L}[]$, will
determine the whole system of linear maps when the category is Cartesian
closed. However, clearly such an ${\cal L}[]$ must satisfy some special
properties which we now outline.

\begin{definition}
A system of maps ${\cal C} \subseteq \X$ is an {\bf exponentiable system
of maps} in case:
\begin{enumerate}[{\bf [ES.1]}]
\item  for each $X,Y,Z \in \X$, all of $1_X,\pi_0: X \x Y \to X,\pi_1: X
\x Y \to Y,\Delta: X \to X \x X$ are in ${\cal C}$ and, furthermore, $\<
f,g\>: X \to Y \x Z \in {\cal C}$ whenever $f,g \in {\cal C}$;
\item  ${\cal C}$ as a system of maps is closed under composition, moreover, whenever $g \in {\cal
C}$ is a retraction and $gh \in {\cal C}$ then $h \in {\cal C}$;
\item  whenever $f:A \to B$ is in ${\cal C}$ and $g: X \to Y$ is any map, then
$g \Rightarrow f: Y \Rightarrow A \to X \Rightarrow B$ is in ${\cal C}$;
\item furthermore
\begin{eqnarray*}
\eta[X] & = & \curry(\pi_1): A \to X \Rightarrow A \\
\mu[X] & = & \curry((\Delta \x 1) (1 \x \eval) \eval): X \Rightarrow (X
\Rightarrow A) \to  X \Rightarrow A
\end{eqnarray*}
are all in ${\cal C}$.
\end{enumerate}
\end{definition}

Here we use the exponential correspondence in the following form:
$$\infer{X \to_{\curry(f)} A \Rightarrow B}{A \x X \to^f B}.$$

\begin{definition} \label{closed-system-defn}
We shall say that a system of \linear maps in a Cartesian closed
category is {\bf closed} in case
\begin{itemize}
\item each evaluation map is linear in its second (higher-order) argument.
\item the system is closed under currying, that is if $f: A \x B \x C \to
D$ is linear in its second argument ($B$) then $\curry(f):  B \x
C \to A \Rightarrow D$ is linear in its first argument.
\end{itemize}
Equivalently, when $\X$ is Cartesian closed, a system of \linear maps is
closed provided $k: B \x X \to Y \in {\cal L}[B]$ if and only if
$\curry(k) \in {\cal L}[]$.
\end{definition}

\begin{proposition}
In a Cartesian closed category $\X$, exponentiable systems of
maps are in bijective correspondence with closed systems of \linear maps.
\end{proposition}

\proof
Given an exponentiable system of maps ${\cal C}$ we define a
system of linear maps $\widetilde{{\cal C}}$, by saying  $f: A \x B \to C$ is
linear in $\X[A]$ if and only if  $\curry(f): B \to A \Rightarrow C$ is
in ${\cal C}$. This defines a system of linear maps since:
\begin{enumerate}[{\bf [LS.1]}]
\item In $\X[A]$ the identity map is $\pi_1: A \x X \to X$ which is
linear as $\curry(\pi_1)$ is.  For the remainder note that $A \x X
\to^{\pi_1} X \to^f Y$ has $\curry(\pi_1 f) = \curry(\pi_0) A
\Rightarrow f$ so that whenever $f$ is in ${\cal C}$ then $\curry(\pi_1
f)$ will be in ${\cal C}$. Finally note that $\curry(\< f,g \>) = \<
\curry(f),\curry(g) \> (\< A \Rightarrow \pi_0,A \Rightarrow \pi_1
\>)^{-1}$.
\item The curry of a composite in $\X[B]$ is given by
$$\curry(\< f,\pi_1 \>g) = \curry(f)(B \Rightarrow \curry(g)) \mu[B]$$
which is a composite of ${\cal C}$-maps when $f$ and $g$ are.

Suppose now that $g \bullet h$ in $\X[B]$ and $\curry(g), \curry(gh) \in
{\cal C}$ with $g$ a retraction (with section $g^{s}$) then $(B
\Rightarrow \curry(g^{s}) \mu[B]$ is right inverse to $(B \Rightarrow
\curry(g)) \mu[B]$ where the latter is in ${\cal C}$.  Also $$B
\Rightarrow \curry(gh)) \mu[B] = B \Rightarrow \curry(g)) \mu[B] B
\Rightarrow \curry(h)) \mu[B]$$ so that $B \Rightarrow \curry(h)) \mu[B]
\in {\cal C}$ But this means $\eta[B](B \Rightarrow \curry(h)) \mu[B] =
\curry(h)$ is in ${\cal C}$.
\item If $\curry(h) \in {\cal C}$ then $\curry(\X[f](h)) = \curry((1 \x
f) h) = (f \Rightarrow B)\curry(h)$ so that the substitution functors
preserve $\widetilde{{\cal C}}$-maps.
\end{enumerate}

We observe that as $\curry(\eval) = 1$ that in $\widetilde{{\cal C}}$
the evaluation map must necessarily be linear in its second
(higher-order) argument.

For the converse, given a system of linear maps ${\cal L}$ in which
evaluation maps are linear in their second argument, we now show  that
${\cal L}[]$ is an exponential system of maps.  The only difficulties
are {\bf [ES.3]} and {\bf [ES.4]}.  For the former note that $f
\Rightarrow g$ is obtained by currying: $$A \x (A' \Rightarrow B')
\to^{f \x 1} A' \x (A' \Rightarrow B') \to^{\eval} B' \to^g B$$ However
observe that this is linear in its second argument so that currying
preserves this establishing {\bf [ES.3]}.  For the latter the argument
is similar as both $\eta$ and $\mu$ are obtained by currying maps which
are linear in the appropriate coordinates.

Finally we must argue that these transitions are inverse: that
$\widetilde{{\cal C}}[] = {\cal C}$ is immediate.  We show that
$\widetilde{{\cal L}[]} = {\cal L}$ by showing ${\cal L} \subseteq
\widetilde{{\cal L}[]}$ and the converse.  Suppose $f \in {\cal L}[A]$
then $\curry(f) \in {\cal L}[]$ so that $f \in \widetilde{{\cal L}[]}
[A]$ as required.   Conversely suppose $f \in \widetilde{{\cal L}[]}[A]$
then $\curry(f) \in {\cal L}[]$ but this means $f = (1 \x \curry(f))
\eval$ is in ${\cal L}[A]$.
\endproof

We note that when ${\cal L}$ is an exponentiable system of linear maps, the
notion of being (${\cal L}$-)bilinear becomes the requirement that both
$\curry(f)$ and $\curry(c_\x f)$ are linear in $\X$.


\subsection{Storage}

A {\bf Cartesian storage category} is a Cartesian category $\X$ together with a
system of linear maps, ${\cal L}$, which is persistently and strongly
classified,\footnote{The notion of classification might (by comparison
with the notion of a subobject classifier in a topos, for example) be
called coclassification---we shall not do that.  Furthermore, we also
use the term ``represented'', especially in the sense that under the
classification given by $S$ and $\varphi$, $f^{\sharp}$ represents $f$.}
in the following sense.

\begin{definition} ~
\begin{enumerate}[(i)]
\item
A system of \linear maps is  {\bf classified}  in
case there is a family of maps $\varphi_X: X \to S(X)$  and for each
$f: X \to Y$  a unique linear map $f^{\sharp}$ such that
$$\xymatrix{X \ar[r]^f \ar[d]_{\varphi_X} & Y \\S(X) \ar@{..>}[ru]_{f^\sharp} }$$
commutes.
\item
A system of \linear maps is {\bf strongly classified}
if there is an object function $S$ and maps $X \to^{\varphi_X} S(X)$ such
that for every $f: A \x X \to Y$ there is  a unique
$f^{\sharp}: A \x S(X) \to Y$ in ${\cal L}[A]$ (that is,
$f^\sharp$ is linear in  its second argument) making
$$\xymatrix{A \x X \ar[d]_{1 \x \varphi_X} \ar[rr]^{f} & & Y \\  A \x S(X)
\ar@{..>}[rru]_{f^{\sharp}}}$$
commute.
\item  A strong classification is said to be {\bf persistent} in case
whenever $f: A \x B \x X \to Y$ is linear in its second argument $B$
then $f^\sharp: A \x B \x S(X) \to Y$ is also linear in its second
argument.
\end{enumerate}
\end{definition}

\begin{remark}
{\em When the \linear maps are classified this makes the inclusion of the linear
maps into a right adjoint.  A strong classification makes the inclusion
of the linear maps into the simple fibration a fibred right adjoint.  A
strong persistent classification allows a powerful proof technique which
we shall use extensively in what follows. To establish the equality of
$$f,g: A \x \underbrace{S(X) \x ...S(X)}_n \to Y$$
when $f$ and $g$ are maps which are linear {\em individually\/} in their
last $n$ arguments, it suffices to show $(1 \x \varphi^n)f = (1 \x
\varphi^n)g$. Note that the linearity precondition is vital! }
\end{remark}

Strong classification gives a morphism of fibrations defined by:
$$S[\X]: \X[\X] \to {\cal L}[\X]; (f,g) \mapsto ( (f \varphi)^\sharp, g).$$
This is a functor on the total category.  Identity maps are preserved because
$$\xymatrix{A \x X \ar[d]_{1 \x \varphi}\ar[r]^{\pi_1} & X \ar[r]^{\varphi} & S(X) \\
            A \x S(X) \ar[rru]_{\pi_1}}$$
commutes making $(\pi_1 \varphi)^\sharp = \pi_1$.  Composition is preserved as
$$\xymatrix{A \x X \ar[d]_{1 \x \varphi}\ar[r]^{\<\pi_0 g,f\>} & B \x Y \ar[d]^{1 \x \varphi} \ar[r]^{f'}
                          & Z \ar[r]^{\varphi} & S(Z) \\
           A \x S(X) \ar[r]_{\< \pi_0 g,(f \varphi)^\sharp \>} & B \x S(Y) \ar[rru]_{(f'\varphi)^\sharp} }$$
Finally $S[\X]$ preserves Cartesian arrows as $S[\X](\pi_1,g) =
(\pi_1,g): (A,S(X)) \to (B,S(X))$.

Therefore the classification is the couniversal property for the
inclusion of the linear maps into the simple fibration, and so
$S$ is a fibred right adjoint.

\begin{remark}{}\rm
We recall some definitions dealing with the notion of strength.

A {\bf functor} $S: \X\to\Y$ between monoidal categories is {\bf strong} if
there is a natural transformation (``strength'') $\theta^S:X\ox S(Y)\to
S(X\ox Y)$ so that the following diagrams commute:
{\small
\[
\xymatrix{\top\ox S(Y) \ar[dr]_{u^L_\ox} \ar[rr]^{\theta} & & S(\top\ox Y) \ar[dl]^{S(u^L_\ox)} \\
 & S(Y)  }
~~~~~~
\xymatrix{X\ox (Y\ox S(Z)) \ar[d]_{1 \ox \theta} \ar[rr]^{a_\ox}  & & (X\ox Y)\ox S(Z)) \ar[dd]^{\theta} \\
                X \ox S(Y \ox Z) \ar[d]_{\theta} \\
                S(X \ox (Y \ox Z)) \ar[rr]_{S(a_\ox)} & & S((X \ox Y) \ox Z)
 }
\]
}

The identity functor is strong (with strength given by the identity),
and if $S$ and $T$ are strong, so is their composite, with strength given by
$\theta^T T(\theta^S): X \ox T(S(Y))\to T(X\ox S(Y))\to T(S(X\ox Y))$.

A {\bf natural transformation} $\psi:S \to T$ between strong functors is
{\bf strong} if the following commutes.

{\small
\[
\xymatrix{ X\ox S(Y) \ar[d]_{1\ox\psi}\ar[r]^{\theta^S} & S (X\ox Y) \ar[d]^{\psi} \\
 X\ox T(Y) \ar[r]_{\theta^T} & T(X\ox Y)
}
\]
}

A {\bf monad} $(S,\eta,\mu)$ is {\bf strong} if each of $S$, $\eta$, and $\mu$
is strong.
\end{remark}

Thus we have proved:

\begin{lemma}
If $\X$ has a strongly classified system of linear maps, then $S:
\X \to \X$ is a strong functor (and thus a morphism of simple
fibrations) and determines a fibred right adjoint to the inclusion
${\cal I}: \partial_{\cal L} \to \partial$.
\end{lemma}

The strength $\theta$ is given explicitly by
$$\xymatrix{A \x X \ar[d]_{1 \x \varphi} \ar[rr]^{\varphi} & & S(A \x X) \\
            A \x S(X) \ar[rru]_{\theta = \varphi^\sharp}}$$
and all the natural transformations are appropriately strong.  The
strength of $\varphi$, the unit of the adjunction, is provided by the
defining diagram of the strength transformation $\theta$ above. Notice
that $\epsilon$ is {\em not\/} natural as $\epsilon f$ is not
necessarily linear unless $f$ is; however, it does satisfy the strength
requirement. However,  $\mu = \epsilon_{S(\_)}$ is natural and also
strong.

This adjunction induces a comonad
$\check{\S} = (S,\delta,\epsilon)$ on the linear maps whose data is given by
$$\xymatrix{A \ar@{=}[rd]\ar[r]^\varphi & S(A) \ar@{..>}[d]^\epsilon \\   & A}
~~~~~~~~
  \xymatrix{A \ar[d]_\varphi \ar[r]^\varphi & S(A) \ar@{..>}[d]^{\delta=S(\varphi)} \\
           S(A) \ar[r]_\varphi & S(S(A))}$$
Notice that by definition both $\epsilon$ and $\delta$ are linear (in ${\cal L}[]$).

In Cartesian storage categories the classification is also persistent, which means
that when you make a function linear in an argument---by classifying
it at that argument---you do not destroy the linearity that any of the
other arguments might have enjoyed.  This has the important consequence
that the universal property can be extended to a multi-classification
property as mentioned above. Here is the argument for the
bi-classification property:

\begin{lemma} \label{bilinear-universal}
In a Cartesian storage category, for each $f: A \x X \x Y \to Z$ there is a
unique $f^{\sharp_2}$, which is linear in both its second and third
arguments ({\em i.e.} bilinear), such that
$$\xymatrix{A \x X \x Y \ar[d]_{1 \x \varphi \x \varphi} \ar[rr]^{f} & & Z \\
            A \x S(X) \x S(Y) \ar@{..>}[rru]_{f^{\sharp_2}}}$$
commutes.
\end{lemma}

\proof
To establish the existence of a map we may extend it in stages:
$$\xymatrix{A \x X \x Y \ar[d]_{1 \x \varphi \x 1} \ar[rr]^{f} & & Z \\
            A \x S(X) \x Y \ar[d]_{1 \x 1 \x \varphi} \ar@{..>}[rru]_{f^{\sharp_1}}\\
            A \x S(X) \x S(Y) \ar@{..>}[rruu]_{f^{\sharp_2}}}$$
the map $f^{\sharp_2}$ is then linear in its last two arguments by
persistence.  Suppose, for uniqueness, that a map $g$, which is linear
in its last two arguments, has $(1 \x \varphi \x \varphi)g =f : A \x
S(X) \x S(Y) \to Z$. Then $f^{\sharp_1} = (1 \x 1 \x \varphi)g$ as the
latter is certainly linear in its middle argument.  Whence $g =
f^{\sharp_2}$.
\endproof

This allows the observation that there is a candidate for monoidal
structure $m_\x:S(A) \x S(B) \to S(A \x B)$ given as the unique bilinear map lifting  $\varphi$:
$$\xymatrix{A \x B \ar[d]_{\varphi \x \varphi} \ar[rr]^{\varphi} & & S(A \x B) \\ S(A) \x S(B) \ar@{..>}[urr]_{m_\x} }$$
We observe that this implies the following two identities:

\begin{lemma} \label{bilinear-identities}
In any Cartesian storage category:
\begin{enumerate}[(i)]
\item  $m_\x = \theta S(\theta') \mu$;
\item  $(\epsilon \x \epsilon) m_\x = m_\x S(m_\x) \epsilon$.
\end{enumerate}
\end{lemma}

\begin{proof} Both can be verified using the universal property outlined in Lemma \ref{bilinear-universal}.  First one notices that
all the maps are bilinear, thus it suffices to show that prefixing with $\varphi \x \varphi$ gives the same maps:
\begin{enumerate}[(i)]
\item  Clearly $m_\x$ is bilinear.  Recall that $\mu = \epsilon_{S(\_)}$. Notice that $\theta S(\theta') \mu$  is linear in its first argument as $\theta$ is and  $S(\theta')\mu$ is linear.
However,  using persistence it is also linear in its second argument as:
$$\xymatrix{A \x S(B) \ar[d]_{\varphi \x 1}  \ar[rd]^{\varphi} \ar[rr]^{\theta'} & & S(A \x B) \ar[dr]^{\varphi} \ar@{=}[rr]  & & S(A \x B) \\
                    S(A) \x S(B) \ar[r]_{\theta} & S(S(A) \x B) \ar[rr]_{S(\theta')} & & S^2(A \x B) \ar[ru]_{~~~\epsilon_{S(A \x B)}= \mu} }$$
This gives:
\begin{eqnarray*}
(\varphi \x \varphi) m_\x & =  & \varphi \\
(\varphi \x \varphi) \theta S(\theta') \mu  & = & (\varphi \x 1)(1 \x \varphi) \theta S(\theta') \mu \\
& = & (\varphi \x 1) \varphi S(\theta') \mu \\
& = & (\varphi \x 1) \theta' \varphi \mu \\
& = & \varphi  \varphi \mu  = \varphi
\end{eqnarray*}
\item  In this case both maps are clearly bilinear and thus the equality is given by:
\begin{eqnarray*}
(\varphi \x \varphi) (\epsilon \x \epsilon) m_\x & = & m_\x \\
(\varphi \x \varphi)  m_\x S(m_\x) \epsilon & = & \varphi S(m_\x) \epsilon \\
& = & m_\x \varphi \epsilon = m_\x
\end{eqnarray*}.
\end{enumerate}
\end{proof}

Then:

\begin{proposition}\label{prop:linmaps}
In a Cartesian storage category:
\begin{enumerate}[(i)]
\item  $\S = (S,\varphi,\mu)$ is a commutative monad, that is $\S$ is a
monoidal monad with respect to the product with $m_\x = \theta
S(\theta') \mu: S(A) \x S(B) \to S(A \x B)$;
\item for every object $A$
$$S(S(X)) \Two^{\epsilon}_{S(\epsilon)} S(X) \to^\epsilon X$$
is an absolute coequalizer;
\item $f: A \x X \to Y$ is \linear in its second argument if and only if
$$\xymatrix{A \x S(X) \ar[d]_{1 \x \epsilon} \ar[r]^\theta &
           S(A \x X)\ar[r]^{~~S(f)} & S(Y) \ar[d]^{\epsilon} \\
            A \x X \ar[rr]_f & & Y}$$
commutes;
\item and so $f: A \x B \to Y$ is bilinear (\/{\em i.e.} linear in its
first two arguments) if and only if
$$\xymatrix{S(A) \x S(B) \ar[d]_{\epsilon \x \epsilon} \ar[r]^{m_{\x}} &
     S(A \x B) \ar[r]^{~~S(f)} & S(Y) \ar[d]^{\epsilon} \\
     A \x B \ar[rr]_f & & Y}$$
commutes.
\end{enumerate}
\end{proposition}

\avoidwidow
\proof ~
\begin{enumerate}[{\em (i)}]
\item  To establish this, given that we know the monad is strong, it
suffices to show that the strength is commutative. That is that the
following diagram commutes:
$$\xymatrix{S(A) \x S(B) \ar[d]_{\theta'} \ar[r]^\theta & S(S(A) \x B) \ar[r]^{S(\theta')} & S(S(A \x B)) \ar[dd]^\mu \\
            S(A \x S(B)) \ar[d]_{S(\theta)} \\
            S(S(A \x B)) \ar[rr]_\mu & & S(A \x B)}$$
The defining diagrams for the two routes round this square are:
$$\xymatrix{A \x B \ar[dr]_{\varphi} \ar[r]^{\varphi \x 1}
                   & S(A) \x B \ar[d]^{\theta'} \ar[dr]^{\varphi} \ar[r]^{1 \x \varphi}
                   & S(A) \x S(B) \ar[d]^\theta \\
            & S(A \x B) \ar@{=}[ddr]\ar[dr]^{\varphi} & S(A \x S(B)) \ar[d]^{S(\theta')} \\
            & & S(S(A \x B)) \ar[d]^\mu \\ & & S(A \x B)}$$
and
$$\xymatrix{A \x B \ar[dr]_{\varphi} \ar[r]^{1 \x \varphi}
                   & A \x S(B) \ar[d]^{\theta} \ar[dr]^{\varphi} \ar[r]^{\varphi \x 1}
                   & S(A) \x S(B) \ar[d]^{\theta'} \\
            & S(A \x B) \ar@{=}[ddr]\ar[dr]^{\varphi} & S(S(A) \x B) \ar[d]^{S(\theta)} \\
            & & S(S(A \x B)) \ar[d]^\mu \\ & & S(A \x B)}$$
Note that, because the classification is persistent, following Lemma
\ref{bilinear-identities} (i),  both these constructed maps  are
bilinear.  Thus, using the universal property of $\varphi \x \varphi$
they are equal.  This shows that the two ways round the square making
the strength commutative are equal.
\item  $\epsilon$ is split by $\varphi$ and so this is a split
coequalizer: that is $\epsilon \varphi = \varphi S(\epsilon)$ and both
$\varphi \epsilon = 1$ and $S(\varphi) \epsilon = 1$.  Thus this is an
absolute coequalizer.
\item
If $f$ is linear in its second argument then $(1 \x \epsilon) f$ is also
linear in its second argument (recall $\epsilon$ is linear). But this
means $f^\sharp = (1 \x \epsilon) f$.  However, also we know $f^\sharp =
\theta S(f) \epsilon$. Thus, the diagram (which shows linearity)
commutes.

Conversely, suppose this diagram commutes; then notice it suffices to
show that we can factor $f^\sharp = (1 \x \epsilon) g$. Then, as $1 \x
\epsilon$ is linear and also a retraction, $g$ must be linear.  Finally
$g$ must of course be $f$ as $f = (1 \x \varphi) f^\sharp = (1 \x
\varphi) (1 \x \epsilon) g = g$. Now to show we can factor $f^\sharp$ in
this manner, we can use the use exactness; it suffices to show that
$$A \x S^2(X) \Two^{1 \x S(\epsilon)}_{1 \x \epsilon} A \x S(X) \to^{f^\sharp} Y$$
commutes.  For this we have the following simple calculation which uses
the fact that $\epsilon\epsilon = S(\epsilon) \epsilon$:
\begin{eqnarray*}
(1 \x \epsilon) f^\sharp
& = & (1 \x \epsilon) \theta S(f) \epsilon \\
& = & (1 \x \epsilon) (1 \x \epsilon) f = (1 \x \epsilon\epsilon) f \\
& = & (1 \x S(\epsilon)\epsilon) f = (1 \x S(\epsilon)) (1 \x \epsilon) f \\
& = & (1 \x S(\epsilon)) \theta S(f) \epsilon = (1 \x S(\epsilon)) f^\sharp.
\end{eqnarray*}
\item This is immediate from the preceding.
\end{enumerate}
\endproof

Notice that this means that in a Cartesian storage category, the system
$\L$ may equivalently be viewed as being induced by the monad $S$.  The
second property says that the induced comonad on the subfibration of
linear maps is an {\em exact modality}.  In the next section, as we study this
comonad in more detail, the significance of this property will become
clear.

We end this subsection with two observations. The first of these is that
the notion of a storage category is simply stable in the sense that each
simple slice of a storage category is itself a storage category.
This, in principle, will allow us to suppress the fibred context of a
statement about storage categories, as such statements will be true in
all simple-slice fibres.

\begin{lemma} \label{simple-slice-context}
If $\X$ is a Cartesian storage category then each simple  slice,
$\X[A]$, is also a Cartesian storage category.
\end{lemma}

The second observation concerns the algebras of the monad $(S,
\varphi,\mu=\epsilon_{S})$.  Recall that an algebra for this monad
is an object $A$ with a map $\nu: S(A) \to A$ such that
$$\xymatrix{ A \ar[r]^{\varphi} \ar@{=}[dr] \ar[r] & S(A) \ar[d]_\nu & A \ar[l]_{S(\varphi)} \ar@{=}[dl] \\ & A}
~~~~~\xymatrix{S^2(A) \ar[d]_{S(\nu)} \ar[r]^{\epsilon_{S(A)}} & S(A) \ar[d]^{\nu} \\
                           S(A) \ar[r]_\nu & A}
$$
Consider the algebras whose structure maps are linear. Of course, there
could be algebras whose structure map does not lie in ${\cal L}$ but
for now we consider the case $\nu \in {\cal L}$.  These algebras, due to
the first triangle identity above, must be precisely of the form
$\nu=\epsilon_A$---and then the remaining identities will automatically
hold.  A homomorphism of such algebras is then precisely a map $f: A \to
B$ which is $\epsilon$-natural.  Thus, we have:

\begin{lemma}
The full subcategory of $S$-algebras, determined by algebras whose
structure maps is linear, is, for any Cartesian storage category $\X$,
isomorphic to the subcategory of linear maps ${\cal L}[]$.
\end{lemma}


\subsection{Abstract coKleisli categories}

Abstract coKleisli categories were introduced (in the dual) by Fuhrmann
\cite{fuhrmann} to provides a direct description of coKleisli
categories. Thus, every coKleisli category is an abstract coKleisli
category and, furthermore, from any abstract  coKleisli category, $\X$,
one can construct a subcategory, $\check{\X}$, with a comonad,
$\check{\S}$, whose coKleisli category is exactly the original category
$\X$.   Of particular interest is when the constructed category,
$\check{\Y}_\S$, of a coKleisli category $\Y_\S$ is precisely $\Y$:
this, it turns out, happens precisely when $\S$ is an exact modality in the sense
above.

This section views a Cartesian storage category from the perspective of its being
an abstract coKleisli category.  A Cartesian storage category is clearly a rather
special abstract coKleisli category so that our aim is to determine what
extra algebraic conditions are being required.

\begin{definition}
An {\bf abstract coKleisli category} is a category $\X$ equipped with a
functor $S$, a natural transformation $\varphi: 1_\X \to S$, and a family of maps
$\epsilon_A: S(A) \to A$ (which is not assumed to be natural) such that
$\epsilon_{S(\_)}$ is natural.  Furthermore:
$$\varphi \epsilon = 1,~~S(\varphi) \epsilon = 1, ~~\mbox{and}~~ \epsilon\epsilon = S(\epsilon)\epsilon$$
An abstract coKleisli category is said to be {\bf Cartesian} if its underlying category, $\X$, is a
Cartesian category such that $!$, $\pi_0$ and $\pi_1$ are all
$\epsilon$-natural (in the sense that $\epsilon \pi_0 = S(\pi_0)\epsilon$, {\em etc.}).
\end{definition}

In an abstract coKleisli category the $\epsilon$-natural maps, that is
those $f$ which satisfy $S(f)\epsilon = \epsilon f$, form a subcategory
$\X_\epsilon$ and on this subcategory $\check{\S} =
(S,S(\varphi),\epsilon)$ is a comonad.  On $\X$, the larger category,
$\S =  (S,\varphi,\epsilon_{S(\_)})$ is a monad.   It is not hard to see
that the $\epsilon$-natural maps form a system of linear maps which are
classified:
$$\xymatrix{X \ar[d]_{\varphi} \ar[rr]^f & & Y \\ S(X) \ar[urr]_{S(f)\epsilon}}$$
where the uniqueness follows since if $h: S(X) \to Y$ has $\varphi h = f$
then $S(f)\epsilon = S(\varphi h)\epsilon = S(\varphi)\epsilon h = h$.

\medskip

In an abstract coKleisli category, we shall always use
the $\epsilon$-natural maps as the default system $\cal L$ of maps.  We now address
the question of when this system is a \linear system of maps.

\medskip

\begin{definition}
An abstract coKleisli category is said to be {\bf strong} in case it is
Cartesian and $S$ is a strong functor, $\varphi$ a strong natural
transformation, and $\epsilon$ is strong, even where it is unnatural,
meaning that the following commutes for all objects $A$ and $X$:
$$\xymatrix{A \x S(X) \ar[r]^\theta \ar[dr]_{1 \x \epsilon} & S(A \x X) \ar[d]^{\epsilon} \\ & A \x X}$$
\end{definition}

In a strong abstract coKleisli category $\X$, $\S = (S,\varphi,\epsilon_{S(\_)})$ is clearly a strong monad.   We observe:

\begin{proposition}
If $(\X,S,\varphi,\epsilon)$ is a strong abstract coKleisli category then the maps
$f: A \x X \to Y$ of $\L[A]\subseteq\X[A]$ such that
$$\xymatrix{ A \x S(X) \ar[d]_{1 \x \epsilon}\ar[r]^\theta & S(A \x X)
\ar[r]^{~~~S(f)} & S(Y) \ar[d]^{\epsilon} \\
             A \x X \ar[rr]_f & & Y}$$
form a system, ${\cal L}_\S[\X]$, of linear maps which are strongly classified by $(S,\varphi)$.
\end{proposition}

\proof  We check the conditions for being a system of linear maps:
\begin{enumerate}[{\bf [LS.1]}]
\item The identity in $\X[A]$ is $\epsilon$-natural as
$$\xymatrix{A \x S(X) \ar[drr]_{\pi_1} \ar[dd]_{1 \x \epsilon} \ar[rr]^\theta & & S(A \x X)  \ar[d]^{S(\pi_1)} \\
                          & & S(X) \ar[d]^\epsilon \\
            A \x X \ar[rr]_{\pi_1} & & X}$$
where we use strength.

We check that the projections in $\X[A]$ are $\epsilon$-natural using:
$$\xymatrix{A \x S(X_1 \x X_2) \ar[d]_{1 \x \epsilon} \ar[r]^\theta &
S(A \x (X_1 \x X_2)) \ar[r]^{~~~S(\pi_1)}
                          & S(X_1 \x X_2) \ar[d]^\epsilon \ar[r]^{~~~S(\pi_i)} & S(X_i) \ar[d]^{\epsilon} \\
            A \x (X_1 \x X_2) \ar[rr]_{\pi_1} & & X_1 \x X_2 \ar[r]_{~~\pi_i} & X_i}$$

Suppose $f_1: A \x X \to Y_1$ and $f_1: A \x X \to Y_1$ are $\epsilon$-natural in $\X[A]$.  Then in the
following diagram the left hand square commutes if and only if the two outer squares corresponding to the
two post compositions with the projections commute.
$$\xymatrix{ A \x S(X) \ar[d]_{1 \x \epsilon} \ar[r]^\theta & S(A \x X) \ar[r]^{S(\< f_1,f_2 \>)}
                   & S(Y_1 \x Y_2) \ar[d]^\epsilon \ar[r]^{~~~S(\pi_i)} & S(Y_i) \ar[d]^\epsilon \\
             A \x X \ar[rr]_{\< f_1,f_2 \>} & & Y_1 \x Y_2 \ar[r]_{~~\pi_i} & Y_i}$$
However the outer squares commute by assumption and so $\< f_1,f_2 \>$ is $\epsilon$-natural in $\X[A]$.
\item We must show that $\epsilon$-natural maps in $\X[A]$ compose. We shall do it explicitly:
$$\xymatrix{ & S(A \x X)  \ar@/^4pc/[rrrd]^{S(\< \pi_0,f \>)} \ar[r]^{S(\Delta \x 1)}
                 & S(A \x A \x X) \ar[drr]^{S(1 \x f)} \\
         A \x S(X) \ar[dd]_{1 \x \epsilon} \ar[ru]^\theta \ar[r]^{\!\!\Delta \x 1}
                 & A \x A \x S(X) \ar[r]^{1 \x \theta} \ar[dd]^{1 \x 1 \x \epsilon}
                 & A \x S(A \x X) \ar[u]^\theta \ar[r]^{~~1\x S(f)} & A \x S(Y) \ar[r]^\theta \ar[dd]^{1 \x \epsilon}
                 & S(A \x Y) \ar[d]^{S(g)} \\
         & &  & & S(Z) \ar[d]^\epsilon \\
         A \x X \ar@/_3pc/[rrr]_{\< \pi_0,f\>} \ar[r]_{\Delta \x 1} & A \x A \x X \ar[rr]_{1 \x f}
                 & & A \x Y \ar[r]_g & Z}$$
Next we must show that if $gh$ is $\epsilon$-natural and $g$ is a retraction which is $\epsilon$-natural then
$h$ is $\epsilon$-natural.  We shall prove it in the basic case leaving it as an exercise for the reader to
do the calculation in a general slice.

We have
$$\xymatrix{S(A)\ar[d]^\epsilon \ar[r]^{S(g)} & S(B)\ar[d]^\epsilon \ar[r]^{S(h)} & S(C) \ar[d]^\epsilon \\
            A \ar[r]^{g} & B \ar[r]^{h} & C}$$
where the left square is known to commute and the outer square commutes.  This means
$S(g) S(h) \epsilon = S(g) \epsilon h$ but $S(g)$ is a retraction as $g$ is so the right square commutes.
\item  We must show that if $f$ is $\epsilon$-natural in $\X[A]$ then $(g \x 1) f$ is $\epsilon$-natural.
Here is the diagram:
$$\xymatrix{A \x S(X) \ar[dr]_{g \x 1} \ar[dd]_\epsilon \ar[r]^\theta
                  & S(A \x X) \ar[rd]_{S(g \x 1)} \ar[rr]^{S((g \x 1)f)} & & S(Y) \ar[dd]^\epsilon \\
            & B\x S(X) \ar[d]_{1 \x \epsilon} \ar[r]_\theta & S(B \x X) \ar[ru]_{S(f)} \\
            A \x X \ar[r]_{g \x 1} & B \x X \ar[rr]_f & & Y}$$
\end{enumerate}
This system of linear maps is always strongly classified by
$(S,\varphi)$ by  defining $f^\sharp = \theta S(f) \epsilon$, which is
clearly $\epsilon$-natural in $\X[A]$, and certainly has $(1 \x
\varphi)\theta S(f) \epsilon = f$.  It is unique, for if $g$ is
$\epsilon$-natural in $\X[A]$ and has $(1 \x \varphi) g = f$, then
$$\theta S(f) \epsilon = \theta S((1 \x \varphi) g)\epsilon = (1 \x S(\varphi)) \theta S(g) \epsilon
                       = (1 \x S(\varphi)) (1 \x \epsilon) g = g.$$
\endproof

To complete the story it remains to establish the property to which persistence corresponds:

\begin{lemma}
A strong abstract coKleisli category has a persistent strong classification of the $\epsilon$-natural maps if and only if $\S$ is a commutative
(or monoidal) monad.
\end{lemma}

\proof
Suppose that $f: A \x X \x Y \to Z$ is linear in its second argument.  That is
$$\xymatrix{A \x S(X) \x Y \ar[d]_{1 \x \epsilon \x 1} \ar[r]^{\theta_2}
                               & S(A \x X \x Y) \ar[r]^{~~~S(f)} & S(Z) \ar[d]^\epsilon \\
            A \x X \x Y \ar[rr]_f & & Z}$$
commutes. Then we must show that after linearizing the last argument the result is still linear in the second
argument. That is the following diagram commutes:
$$\xymatrix{ & S(A \x X \x S(Y)) \ar[r]^{~~S(\theta_3)} & S^2(A \x X \x
Y) \ar[dr]^\epsilon \ar[r]^{~~~S^2(f)}
                                    & S^2(Z) \ar[dr]^\epsilon \\
            A \x S(X) \x S(Y) \ar[ur]^{\theta_2} \ar[d]_{1 \x \epsilon \x 1} \ar[dr]_{\theta_3} & &
                            & S(A \x X \x Y) \ar[r]^{~~~~S(f)} & S(Z) \ar[d]^\epsilon \\
            A \x X \x S(Y) \ar[dr]_{\theta_3} & S(A \x S(X) \x Y)
\ar[d]^{S(1 \x \epsilon \x 1)} \ar[r]^{~~S(\theta_2)}
                           & S^2(A \x X \x Y) \ar[r]_{~~~S^2(f)} \ar[ur]_\epsilon
                           & S^2(Z) \ar[ur]_\epsilon \ar[d]_{S(\epsilon)}
                           & Z  \\
            & S(A \x X \x Y) \ar[rr]_{S(f)} & & S(Z) \ar[ur]_\epsilon}$$

 Conversely we have already established that a persistent strong classification gives rise to a commutative monad.
\endproof

We shall say that a strong abstract coKleisli category is {\bf
commutative} if its monad is.  We have now established:

\begin{theorem}
If $\X$ is a Cartesian category, then $\X$ is a Cartesian storage category (that
is, $\X$ has a persistently strongly classified system of \linear maps)
if and only if it is a strong abstract commutative coKleisli category.
\end{theorem}


\subsection{Linear idempotents}

A very basic construction on a category is to split the idempotents and we pause our development to briefly consider how this construction
applies to storage categories.  In a storage category there is a slightly subtle issue as it is quite possible for a linear idempotent to split into two
non-linear components.   This means that even though an idempotent may split in the ``large'' category of all maps it need not in
the subcategory of linear maps.   We shall be interested in splitting the linear idempotents linearly, for in this case we can show that the
resulting category is also a storage category which extends the original category without introducing any new linear maps.

We start with a basic observation:

\begin{lemma}  \label{storage-splitting}
In a Cartesian storage category:
\begin{enumerate}[(i)]
\item When $fg$ is linear and $g$ is monic and linear, then $f$ is linear;
\item When $fg$ is linear and $f$ is linear with $S(f)$ epic, then $g$ is linear;
\item A linear idempotent $e:X \to Y$ splits linearly, in the sense that there is a splitting $(r,s)$ with $rs = e$ and $sr = 1$ with both $r$ and $s$ linear
iff there is a splitting in which at least one of $r$ and $s$ are linear.
\end{enumerate}
\end{lemma}

\begin{proof} Consider
\[
\xymatrix{ S(X) \ar[rr]^{S(fg)} \ar[dd]_{\epsilon} \ar[dr]^{S(f)}  && S(Y) \ar[dd]^{\epsilon}  \\
 & S(Y) \ar[dd]^{\raisebox{20pt}{\scr$\epsilon$}} \ar[ur]^{S(g)}  \\
X \ar[rr]|{~}^{fg~~~~~~~~~~~} \ar[dr]_{f}
  && Y  \\
 & Y \ar[ur]_{g} }
\]
Under the conditions of {\em (i)} the back face commutes since $fg$ is linear
({\em i.e.} $\epsilon$-natural), and the ``$g$ face'' commutes.  Then $S(f)\epsilon g = S(f) S(g) \epsilon = \epsilon f g$ and, thus, as
$g$ is monic, $S(f)\epsilon = \epsilon f$  implying $f$ is linear.

Under the conditions of {\em (ii)} the back face still commutes and the ``$f$ face'' commutes.  So then  $S(f)\epsilon g = \epsilon f g =  S(g) F(g) \epsilon$ and, thus, as
$S(f)$ is epic $\epsilon g = S(g) \epsilon$  implying $g$ is linear.

Finally for {\em (iii)} if $e$ splits and $s$ is linear then $e=rs$ satisfies {\em (i)} while if $r$ is linear it satisfies {\em (ii)} as $S(r)$ being a retraction is certainly epic.
\end{proof}

In a storage category we shall often wish to split all the linear idempotents, that is split all idempotents $e$ such that $e \in {\cal L}$.  An important observation
is that this can be done entirely formally:

\begin{proposition}
\label{linear-idempotent-splitting}
If $\X$ is a Cartesian storage category then the category obtained by splitting the linear idempotents, ${\sf Split}_{\cal L}(\X)$, is a Cartesian storage category in which linear idempotents split and the embedding
$${\cal I}: \X \to {\sf Split}_{\cal L}(\X)$$
preserves and reflects linear maps and the classification.
\end{proposition}

\proof
The objects of ${\sf Split}_{\cal L}(\X)$  are the linear idempotents in $\X$,  and the maps $f: e \to e'$ are, as usual, those such that $efe'=f$.  As linear
idempotents are closed under products
(that is whenever $e$ and $e'$ are linear idempotents then $e \x e'$ is a linear idempotent) it is standard that these form a Cartesian category.   We shall say that $f: e \x e' \to e''$ is linear in its first argument precisely when it is so as a map in $\X$.  It is then immediate that in ${\sf Split}_{\cal L}(\X)$ linear idempotents will split linearly.

The functor $S$ is now defined on ${\sf Split}_{\cal L}(\X)$ by taking an object $e$ to $S(e)$ and a map $f$ to $S(f)$.   The transformations $\varphi$ and $\epsilon$ can also be extended by setting $\varphi_e := e \varphi S(e) =  \varphi S(e) = e \varphi$ and $\epsilon_e:= S(e) \epsilon e = \epsilon e = S(e)\epsilon$ (where here we use the linearity of $e$).

To show that $(S,\varphi)$ on ${\sf Split}_{\cal L}(\X)$ classifies linear maps in ${\sf Split}_{\cal L}(\X)$ consider
$$\xymatrix { a \x e \ar[rr]^f \ar[d]_{a \x \varphi} & & e' \\ a \x S(e) \ar@{..>}[urr]_{f^\sharp} }$$
where $a$, $e$, $e'$ are linear idempotents and $f= (a \x e)fe'$ and $f^\sharp$ is the linear classification for $f$ in $\X$ then it suffices to show that
$f^\sharp = (a \x S(e))f^\sharp e'$ to show this is a classification in ${\sf Split}_{\cal L}(\X)$.   To show this first note that $(a \x S(e))f^\sharp e'$ is linear, so if
$(1 \x \varphi)(a \x S(e))f^\sharp e' =f$  we will have secured the equality.  For this we note:
$$(1 \x \varphi)(a \x S(e))f^\sharp e' = (a \x e)(1 \x \varphi) )f^\sharp e' = (a \x e) f e' = f.$$
Persistence now follows immediately from persistence in $\X$.

We also need to show that the linear maps in ${\sf Split}_{\cal L}(\X)$ satisfy {\bf [LS.1]}, {\bf [LS.2]} ,and {\bf [LS.3]}.  The only difficulty concerns {\bf [LS.2]} and the condition on linear
retractions.  In a slice ${\sf Split}_{\cal L}(\X)[a]$ we must show that if $g$ is a linear retraction (where the section is not necessarily linear) and $gh$ is linear then $h$ is linear.   This
gives the following commuting diagram of maps
$$\xymatrix{a \x e' \ar[d]_{a \x e'} \ar[rr]^{\< \pi_0,v \>} & & a \x e \ar[dll]^{\< \pi_0,g \>} \ar[d]^{\< \pi_0,g\> h} \\ a \x e' \ar[rr]_h & & e''}$$
in which the downward pointing arrows contain the known linear
components and $v$ is the section of $g$  which gives rise to the
leftmost arrow being an ``identity map'' at the idempotent $a \x e'$. We
must show that $h$ is linear in this slice, that is that $(1 \x
\epsilon) h = \theta S(h) \epsilon$.   Here is the calculation:
 \begin{eqnarray*}
 \theta S(h) \epsilon & = & \theta S((a \x e')h) \epsilon ~~~~~~~~~~~\mbox{as $h: a \x e' \to e''$}\\
 & = & \theta S(\<\pi_0,v\>\<\pi_0,g\>h) \epsilon ~~~~\mbox{as $v$ is a section of $g$}\\
 & = & \< \pi_0,\theta S(v) \> \theta S(\<\pi_0,g\>h) \epsilon \\
 & = &  \< \pi_0,\theta S(v) \> (1 \x \epsilon) \<\pi_0,g\>h ~~~\mbox{as $gh$ is linear} \\
 & = & \< \pi_0,\theta S(\< \pi_0,v  \>g) \> (1 \x \epsilon) h ~~~\mbox{as $g$ is linear}\\
 & = & \< \pi_0,\theta S(\pi_1 e') \>(1 \x \epsilon) h ~~~~~\mbox{as $v$ is a section of $g$}\\
 & = & (1 \x \epsilon)(1 \x e')h = (1 \x \epsilon)h.
 \end{eqnarray*}
 \endproof

 It is worth remarking that splitting linear idempotents of $\X$ does not cause linear idempotents in $\X[A]$ to split.  One can, of course, split the linear
idempotents in $\X[A]$ but this in general will have more objects than ${\sf Split}_{\cal L}(\X)[A]$ as there will be more idempotents which are ``linear in one
argument'' than are just linear.  More precisely, a linear idempotent in
 $\X[A]$ is an $e': A \x X \to X$ linear in $X$ with $\<\pi_0,e'\>e' = e'$ and these, in general, will strictly include idempotents $\pi_1e$ where $e$ is a linear idempotent in $\X$.

\subsection{Forceful comonads}

At this stage we have a fairly good grasp of what the monad of a
Cartesian storage category looks like.  However, notably absent has been
a discussion of the properties of the comonad $\check{\S}$ on the linear
maps of a Cartesian storage category.  In this section we correct this
defect.

Recall that the strength in a Cartesian storage category is not linear,
in general, so that the comonad $\check{\S}$ will not have a strength.
To guarantee that $\check{\S}$ be strong,
we shall postulate the presence of a map which generates a strength
map in the coKleisli category: this we call a {\em force\/} and we
develop its properties below.

\begin{definition}
A {\bf (commutative) force} for a comonad $\check{\S} =
(S,\epsilon,\delta)$ is a natural transformation
$$S(A \x S(X)) \to^{\psi} S(A \x X)$$
which renders the following diagrams commutative:
\begin{enumerate}[{\bf [{Force}.1]}]
\item Associativity of force:
$$\xymatrix{S(A \x (B \x S(C)) \ar[d]_{\delta} \ar[rr]^{a_\x} & & S((A \x B) \x S(C)) \ar[dddd]^{\psi} \\
            S^2(A \x (B \x S(C)) \ar[d]_{S(\sigma_\x)} \\
            S(S(A) \x S(B \x S(C)) \ar[d]_{S(\epsilon \x \psi)}\\
            S(A \x S(B \x C)) \ar[d]_{\psi} \\
            S(A \x (B \x C)) \ar[rr]_{S(a_\x)} & & S((A \x B) \x C)
           }$$
\item  Projection and force:
$$\xymatrix{ S(A \x S(B)) \ar[d]_{S(\pi_1)} \ar[rr]^{\psi} & & S(A \x B) \ar[d]^{S(\pi_1)} \\
             S^2(B) \ar[rr]_{\epsilon} & & S(B)}$$
\item Forceful naturality:
$$\xymatrix{S(A \x S(B)) \ar[rr]^{\delta} \ar[dd]_{\psi} & & S^2(A \x S(B)) \ar[d]^{S(\sigma_\x)} \\
                        & & S(S(A) \x S^(B)) \ar[d]^{S(1 \x (\epsilon\delta))}  \\
                        S(A \x B) \ar[d]_{\delta} & & S(S(A) \x S^2(B)) \ar[d]^{\psi} \\
                        S^2(A \x B) \ar[rr]_{S(\sigma_\x)} & & S(S(A) \x S(B)) }$$
\item The forcefulness of counit:
$$\xymatrix{S(A \x S^2(B)) \ar[d]_{\psi} \ar[rr]^{S(1 \x \epsilon)} & & S(A \x S(B)) \ar[d]^{\psi} \\
            S(A \x S(B)) \ar[rr]_{\psi} & & S(A \x B) }$$
\item The forcefulness of comultiplication:
$$\xymatrix{S(A \x B) \ar[r]^{\delta} \ar@{=}[d] & S^2(A \x B)
\ar[r]^{\!\!\!\!\!S(\sigma_\x)} & S(S(A) \x S(B)) \ar[d]^{\psi} \\
            S(A \x B) & & S(S(A) \x B) \ar[ll]^{S(\epsilon \x 1)}}$$
\item Commutativity of force:
$$\xymatrix{S(S(A) \x S(B)) \ar[d]_{\psi'} \ar[rr]^{\psi} & & S(S(A) \x B) \ar[d]^{\psi'} \\
            S(A \x S(B)) \ar[rr]_{\psi} & & S(A \x B)}$$
where $\psi': = S(c_\x) \psi S(c_\x)$ is the symmetric dual of the force.
\end{enumerate}
\end{definition}

\noindent
We shall say that a comonad is {\bf forceful} if it comes equipped with
a force.  First we observe the following.

\begin{proposition}
In any Cartesian storage category $\X$ the comonad, $\check{\S}$ on the linear maps has a force.
\end{proposition}

\proof
If $\X$ is a Cartesian storage category we may define $\psi = S(\theta)\epsilon$:
by inspection this is a linear map ({\em i.e.} in ${\cal L}[]$).
We must check it satisfies all the given properties.
\begin{enumerate}[{\bf [{Force}.1]}]
\item The interaction of associativity and strength gives the equation $(1 \x \theta)\theta S(a_\x) = a_\x \theta$. We shall use this and the
fact that $S(\varphi h)\epsilon = h$ for linear $h: S(X) \to Y$ in the calculation:
\begin{eqnarray*}
\delta S(\sigma_\x) S(\epsilon \x \psi) \psi S(a_\x)
    & = & S( \varphi \delta S(\sigma_\x) S(\epsilon \x \psi) \psi S(a_\x)) \epsilon \\
    & = & S( \varphi \varphi  S(\sigma_\x) S(\epsilon \x \psi) \psi S(a_\x)) \epsilon \\
    & = & S( \varphi \sigma_\x (\epsilon \x \psi) \varphi \psi S(a_\x)) \epsilon \\
    & = & S( (\varphi \x \varphi) (\epsilon \x \psi) \varphi \psi S(a_\x)) \epsilon \\
    & = & S((1 \x \theta)\theta S(a_\x)) \epsilon \\
    & = & S(a_\x \theta) \epsilon = S(a_\x) \psi
\end{eqnarray*}
\item The projection of force:
\begin{eqnarray*}
\psi S(\pi_1) & = & S(\theta)\epsilon S(\pi_1) \\
              & = & S(\theta S(\pi_1)) \epsilon \\
              & = & S(\pi_1) \epsilon
\end{eqnarray*}
\item Forceful naturality:
\\
We shall use the couniversal properties repeatedly:  to show $\psi \delta S(\sigma_\x) = \delta S(\sigma_\x(1 \x \epsilon\delta))\psi$ it suffices,
as they are both linear to show that they are equal where prefixed with $\varphi$:
\begin{eqnarray*}
\varphi \psi \delta S(\sigma_\x) & = & \theta \delta S(\sigma_\x) \\
                   & = & (1 \x \delta) \theta S(\theta \sigma_\x) \\
\varphi \delta S(\sigma_\x (1 \x (\epsilon\delta))) \psi & = & \varphi \varphi S(\sigma_\x (1 \x (\epsilon\delta))) \psi \\
                   & = & \varphi (\sigma_\x (1 \x (\epsilon\delta))) \theta \\
                   & = & (\varphi \x \varphi) (1 \x (\epsilon\delta))) \theta \\
                   & = & (\varphi \x \delta) \theta
\end{eqnarray*}
The resulting expressions are both linear in the second argument so that they are equal provided precomposing with $1 \x \varphi$ makes them
equal:
\begin{eqnarray*}
(1 \x \varphi) (1 \x \delta) \theta S(\theta \sigma_\x)
       & = & (1 \x (\varphi\varphi)) \theta S(\theta \sigma_\x) \\
       & = & (1 \x \varphi) \varphi S(\theta \sigma_\x) \\
       & = & (1 \x \varphi) \theta \sigma_\x \varphi \\
       & = & \varphi \sigma_\x \varphi = (\varphi \x \varphi)\varphi \\
(1 \x \varphi)  (\varphi \x \delta) \theta
       & = & (\varphi \x \varphi)(1 \x \varphi) \theta \\
       & = &  (\varphi \x \varphi)\varphi
\end{eqnarray*}
This establishes the equality!
\item For the forcefulness of the counit both sides are linear as so
equal if they are equal after precomposing with $\varphi$. However,
$\varphi\psi\psi = \theta\psi$ and $\varphi S(1 \x \epsilon)\psi = (1 \x
\epsilon) \varphi\psi = (1 \x \epsilon) \theta$.  These last are linear
in their second argument and, therefore, are equal if and only if
precomposing with $1 \x \varphi$ makes them equal: $$(1 \x \varphi)
\theta \psi = \varphi \psi = \theta = (1 \x \varphi) (1 \x \epsilon)
\theta$$
\item The forcefulness of the comultiplication:
\begin{eqnarray*}
\delta S(\sigma_\X) \psi S(\epsilon \x 1) & = & \delta S(\sigma_\x \theta) \epsilon S(\epsilon \x 1) \\
                                          & = & \delta S(\sigma_\x \theta S(\epsilon \x 1)) \epsilon \\
                                          & = & S(\varphi \sigma_\x (\epsilon \x 1) \theta) \epsilon \\
                                          & = & S((\varphi \x \varphi) (\epsilon \x 1) \theta) \epsilon \\
                                          & = & S((1 \x \varphi) \theta) \epsilon \\
                                          & = & S( \varphi ) \epsilon = 1
\end{eqnarray*}
\item Commutativity of the force:
\begin{eqnarray*}
\psi \psi' & = & S(\theta) \epsilon S(\theta') \epsilon \\
           & = & S(\theta S(\theta') \epsilon \epsilon \\
           & = & S(\theta'S(\theta)) \epsilon \epsilon \\
           & = & S(\theta')\epsilon S(\theta) \epsilon \\
           & = & \psi' \psi
\end{eqnarray*}
\end{enumerate}

\endproof

\noindent Conversely:

\begin{proposition}
\label{abstract-coKleisli-basic}
Let $\S$ be a forceful comonad on $\X$ a Cartesian category then the
coKleisli category $\X_\S$ is a Cartesian storage category.
\end{proposition}

\proof
The coKleisli category is immediately an abstract coKleisli category (see \cite{fuhrmann}): it remains only to show that
all the data ($S$, $\varphi$, and $\epsilon$) are strong, and that the induced monad is commutative.

The following is the interpretation for the data of the Cartesian storage category:
\begin{eqnarray*}
\llbracket \varphi \rrbracket & := & 1_{S(A)} \\
\llbracket \epsilon \rrbracket & := & \epsilon\epsilon \\
\llbracket S( f ) \rrbracket & := & \epsilon\delta S(\llbracket f \rrbracket) \\
\llbracket \theta \rrbracket & := & \psi \\
\llbracket f  \x g \rrbracket & := & \< S(\pi_0)\llbracket f \rrbracket,S(\pi_1)\llbracket g \rrbracket \>
                                 = \sigma_\x (\llbracket f \rrbracket \x \llbracket g \rrbracket) \\
\llbracket \Delta \rrbracket & := & \epsilon \Delta \\
\llbracket \pi_i \rrbracket & := &\epsilon \pi_i \\
\llbracket a_\x \rrbracket & := & \epsilon a_\x \\
\llbracket f g \rrbracket & := & \delta S(\llbracket f \rrbracket)\llbracket g \rrbracket.
\end{eqnarray*}

First we must show that $\bm\theta$ satisfies the requirements of being a strength
for ${\bf S}$ with respect to the product.  This may be expressed as following two requirements and naturality:
$$\xymatrix{A \x S(X) \ar[r]^{\theta} \ar[rd]_{\pi_1} & S(A \x X) \ar[d]^{ S(\pi_1)} \\
            & S(X)}
~~~~~~~~
\xymatrix{A \x (B \x S(X)) \ar[d]_{a_\x} \ar[rr]^{1 \x \theta}
           & & A \x S(B \x X) \ar[r]^{\theta}  & S(A \x (B \x X)) \ar[d]^{S(a_\x)} \\
          (A \x B) \x S(X) \ar[rrr]_{\theta}  & & & S((A \x B) \x X)}$$
The two diagrams are verified by the following calculation:
\begin{eqnarray*}
\llbracket \theta \bullet S(\pi_1) \rrbracket
  & = & \delta S(\psi) \epsilon \delta S(\epsilon \pi_1) \\
  & = & \psi S(\pi_1) = S(\pi_1) \epsilon = \epsilon \pi_1 \\
  & = & \llbracket \pi_1 \rrbracket \\
\llbracket  (1 \x \theta) \bullet \theta \bullet S(a_\x) \rrbracket
  & = & \delta S(\< S(\pi_0)\epsilon,S(\pi_1) \psi \>) \delta S(\psi) \epsilon \delta S(\epsilon a_\x) \\
  & = & \delta S(\< S(\pi_0)\epsilon,S(\pi_1) \psi \>) \psi S(a_\x)  \\
  & = & S(a_\x) \psi \\
  & = & \llbracket a_\x \bullet \theta \rrbracket
\end{eqnarray*}
For naturality we have:
\begin{eqnarray*}
\llbracket (f \x S(g)) \theta \rrbracket & = & \delta S(\sigma_\x(f \x (\epsilon\delta S(g)))) \psi \\
             & = & \delta S(\sigma_\x(1\x (\epsilon\delta))) S(f \x S(g)) \psi \\
             & = &  \delta S(\sigma_\x(1\x (\epsilon\delta))) \psi S(f \x g) \\
             & = & \psi \delta S(\sigma_\x) S(f \x g) \\
             & = & \delta S(\psi) \epsilon\delta S(\sigma_\x(f \x g)) = \llbracket \theta S(f \x g) \rrbracket
\end{eqnarray*}
Next we must verify that both $\bm\varphi$ and $\bm\epsilon_S$
({\em i.e.} $\bm\mu$) are strong:
$$\xymatrix{A \x X  \ar[drr]_{ \varphi} \ar[rr]^{ 1 \x \varphi} & & A \x S(X) \ar[d]^{ \theta} \\
            & & S(A \x X) }
~~~~~~~~
\xymatrix{A \x S^2(X) \ar[d]_{ 1 \x \epsilon_S} \ar[r]^{ \theta} &
    S(A \x S(X)) \ar[r]^{S(\theta)} & S^2(A \x X)  \ar[d]^{ \epsilon_S} \\
     A\x S(X) \ar[rr]_{\theta}  & & S(A \x X) }
$$
\begin{eqnarray*}
\llbracket (1 \x \varphi) \theta \rrbracket
  & = & \delta S(\sigma_\x (\epsilon \x 1)) \psi \\
  & = & \delta S(\sigma_\x) \psi S(\epsilon \x 1) \\
  & = & 1 = \llbracket \varphi \rrbracket \\
\llbracket (1 \x  \epsilon) \theta \rrbracket
  & = & \delta S(\sigma_\x (\epsilon \x \epsilon \epsilon)) \psi \\
  & = & \delta S(\sigma_\x (\epsilon \x \epsilon)) S(1 \x \epsilon) \psi \\
  & = & \delta S(\sigma_\x (\epsilon \x \epsilon)) \psi \psi \\
  & = & \psi \psi \\
  & = & \delta S(\psi\delta S(\psi)\epsilon)\epsilon  \\
  & = & \delta S(\psi)\delta S(\epsilon\delta S(\psi))\epsilon\epsilon  \\
  & = & \llbracket \theta S( \theta) \epsilon \rrbracket
\end{eqnarray*}
where we used the following equality
$\delta S(\sigma_\x (\epsilon \x \epsilon)) = 1_{S(A \x B)}$ which holds as
$$\delta S(\sigma_\x (\epsilon \x \epsilon))
       = \delta S(\<S(\pi_0),S(\pi_1)\>(\epsilon \x \epsilon))
       = \delta S(\<\epsilon\pi_0,\epsilon\pi_1\>))
       = \delta S(\epsilon \<\pi_0,\pi_1\>)
       = \delta S(\epsilon) = 1.$$

Lastly we must show that the induced monad is commutative, that is:
$$\xymatrix{ S(A) \x S(B) \ar[d]_{\theta'} \ar[rr]^{\theta} & & S(S(A) \x B) \ar[rr]^{S(\theta')} & & S^2(A \x B) \ar[d]^{\mu} \\
             S(A \x S(B)) \ar[rr]_{S(\theta)} & & S^2(A \x B) \ar[rr]_{\mu} & & S(A \x B)}$$
\begin{eqnarray*}
\llbracket \theta S(\theta') \mu \rrbracket
    & = & \delta S(\psi) (\delta S(\epsilon\delta S(\psi')) \epsilon\epsilon \\
    & = & \delta S(\psi \delta S(\psi') \epsilon) \epsilon \\
    & = & \psi \psi' = \psi'\psi \\
    & = & \delta S(\psi' \delta S(\psi) \epsilon) \epsilon \\
    & = & \delta S(\psi') (\delta S(\epsilon\delta S(\psi)) \epsilon\epsilon \\
    & = & \llbracket \theta' S(\theta) \mu \rrbracket .
\end{eqnarray*}
\endproof

We shall call a Cartesian category with a forceful comonad a {\bf
Cartesian linear category}.  We may summarize the
above results as follows:

\begin{theorem}
\label{exact-comonad}
A category is a Cartesian storage category if and only if it is the
coKleisli category of a Cartesian linear category. A Cartesian linear category
is the linear maps of a Cartesian storage category if and only if its comonad is exact.
\end{theorem}

Recall that a comonad is exact when the commuting diagram:
   $$S(S(X)) \Two^{S(\epsilon)}_{\epsilon} S(X) \to^{\epsilon} X$$
is a coequalizer. A category with an exact comonad is always the
subcategory of $\epsilon$-natural maps of its coKleisli category.  This
allows the original category to be completely recovered from the
coKleisli category.

\medskip

Starting with a forceful comonad $\S$ on a Cartesian category $\X$ one can directly form the
simple fibraton of the coKleisli category.  The total category can be described by:
\begin{description}
\item{{\bf Objects:}}  $(A,X)$ where $A,X \in \Y$;
\item{{\bf Maps:}} $(f,g): (A,X) \to (B,Y)$ where $f: S(A) \to B$ and $g:S(A \x X) \to Y$;
\item{{\bf Identities:}} $(\epsilon,S(\pi_1)\epsilon): (A,X) \to (A,X)$;
\item{{\bf Composition:}} $(f,g)(f',g') = (\delta S(f) f',\delta S(\<S(\pi_0)f,g\>)g')$.
\end{description}
That this amounts to the simple fibration $\partial: \X_\S[\X_\S] \to \X_\S; (f,g) \mapsto f$ is easily checked.

\begin{remark}{}
\rm
One might expect for there to be more interactions between the comultiplication and the force.  In fact, there are
but they are often not so obvious!  Here is an example:
$$\xymatrix{S(A \x S(B))  \ar[d]_{\psi} \ar[rr]^{S(1 \x \delta)} & & S(A \x S^2(B)) \ar[d]^{\psi} \\
            S(A \x B)  & & S(A \x S(B)) \ar[ll]^{\psi}}$$
An easy, but perhaps rather unsatisfactory way to check this is by
looking in the corresponding Cartesian storage category where we have
rather easily:
\begin{eqnarray*}
S(1 \x \delta) \psi \psi & = & S(1 \x S(\varphi)) S(\theta) \epsilon S(\theta) \epsilon \\
                         & = & S((1 \x S(\varphi))\theta S(\theta)) \epsilon \epsilon \\
                         & = & S(\theta S((1 \x \varphi)\theta) \epsilon) \epsilon  \\
                         & = & S(\theta \varphi \epsilon) \epsilon = S(\theta)\epsilon = \psi
\end{eqnarray*}
However, one might reasonably want a direct proof.  We shall use the following equality
$$\delta S(\sigma_\x (\epsilon \x \epsilon)) = 1_{S(A \x B)}$$
which holds since
$$\delta S(\sigma_\x (\epsilon \x \epsilon))
       = \delta S(\<S(\pi_0),S(\pi_1)\>(\epsilon \x \epsilon))
       = \delta S(\<\epsilon\pi_0,\epsilon\pi_1\>))
       = \delta S(\epsilon \<\pi_0,\pi_1\>)
       = \delta S(\epsilon) = 1.$$
Here is the calculation:
\begin{eqnarray*}
\psi & = & S(1 \x \delta S(\epsilon))\psi = (1 \x \delta) \psi S(1 \x \epsilon) \\
   & = & (1 \x \delta) \psi \delta S(\sigma_\x) \psi S(\epsilon \x 1) S(1 \x \epsilon) \\
   & = & (1 \x \delta) \psi \delta S(\sigma_\x) \psi S(\epsilon \x \epsilon) \\
   & = & (1 \x \delta) \delta S(\sigma_\x) S(1 \x (\epsilon\delta)) \psi \psi S(\epsilon \x \epsilon) \\
   & = & (1 \x \delta) \delta S(\sigma_\x) S(1 \x (\epsilon\delta)) \psi S(\epsilon \x S(\epsilon)) \psi \\
   & = & (1 \x \delta) \delta S(\sigma_\x) S(1 \x (\epsilon\delta)) S(\epsilon \x S^2(\epsilon)) \psi \psi \\
   & = & (1 \x \delta) \delta S(\sigma_\x) S(\epsilon \x (\epsilon\delta S^2(\epsilon))) \psi \psi \\
   & = & (1 \x \delta) \delta S(\sigma_\x)(\epsilon \x \epsilon) (1 \x (S(\epsilon) \delta)) \psi \psi \\
   & = & (1 \x \delta) (1 \x (S(\epsilon) \delta)) \psi \psi \\
   & = & (1 \x \delta) \psi \psi
\end{eqnarray*}
\end{remark}

\bigskip

We now have the following which summarizes the main results of this section

\begin{theorem}
For a Cartesian category $\X$, the following are equivalent:
\begin{itemize}
\item $\X$ is a Cartesian storage category;
\item $X$ is a strong abstract commutative coKleisli category;
\item $X$ is the coKleisli category of a Cartesian category with a forceful comonad.
\end{itemize}
\end{theorem}



\section{Tensor storage categories}  
\label{tensor-storage-cats}

Our objective is now to link our development to the categorical semantics of linear
logic through the ``storage categories''  which were developed from the ideas introduced in \cite{Seely}.
The original definition of these categories required a comonad $S$ on a $*$-autonomous
category and natural isomorphisms $s_\x: S(A \x B) \to S(A) \ox S(B)$
and $s_1: S(1) \to \top$.  Subsequently Gavin Bierman realized that, in order
to ensure that equivalent proofs of Multiplicative Exponential Linear Logic (MELL)
were sent to the same map, the comonad actually had to be monoidal.  Bierman called these
categories, in which this new requirement was added, ``new Seely'' categories.

Bierman's examination of MELL also revealed that, even in the
absence of additives, the coKleisli category of a symmetric monoidal
category with an exponential would have products and, furthermore, when
the original category was closed, would be Cartesian closed.  Thus, it
seemed that the additive structure was not really necessary for the
theory. Bierman called the categories which provided models for MELL
{\em linear categories}.  In this manner, the additive structure, basic in Seely's
original definition, was relegated to a secondary and largely optional
role.

Andrea Schalk \cite{Schalk} collected the
various axiomatizations of Seely categories originating from this work,
and removed the requirement that the category be closed.  She  showed that
this was an orthogonal property whose main purpose was to ensure the
coKleisli category was Cartesian closed.  She called the modalities {\em
linear exponential comonads}, and thus replaced linear
categories with symmetric monoidal categories with a linear exponential
comonad.

Even more recently Paul-Andr\'e Melli\`es \cite{Mellies} while revisiting
``categorical models of linear logic'' concentrated entirely upon the
exponential structure: notably $*$-autonomous categories hardly rated a
mention and the additives are reduced to at most products---although
closedness is assumed throughout.  Of note was the emphasis that was
placed on the role of the monoidal adjunction which was induced between
the linear and Cartesian categories.  Of particular interest was an
axiomatization for Seely's original ideas which showed that rather than
demand that the comonad be a monoidal comonad one could obtain the same
effect by axiomatizing the Seely isomorphism itself more carefully. Here
we shall follow this idea and, adapting it somewhat, obtain a convenient
description of the variety of Seely category in which we are interested.
These are ``new Seely categories'', but following Andrea Schalk's lead, the
requirement of closedness is dropped.

One thing that should be emphasized is that in, this exposition of Seely
categories, we are focusing on the product structure: the monoidal
structure is regarded as secondary, indeed, even generated by products.
This is absolutely the opposite to the general trend in the work cited
above where an underlying theme is to decompose the product structure
into a more fundamental tensorial structure.  Mathematically, of course,
there is no tension between these approaches, as in these settings the
product and tensor are linked and it should be no surprise that the
linkage can be worked in both directions.  However, there is perhaps a
philosophical message, as it challenges the precept of what should be
taken as primary. Seemingly in spite of this, we shall call the current
notion ``tensor storage categories'', as they are very similar to the
``storage categories'' we considered in \cite{diffl}, in the context of
(tensor) differential categories.

The section starts by providing  an exposition of tensor storage categories: we start
from the definition of a storage transformation and explain why this is
the same as a coalgebra modality.  A tensor storage category is then defined to be
a symmetric monoidal category with products and a storage transformation
which is an isomorphism.  We then prove that this implies that the
modality is a monoidal comonad and, thus, that we do indeed obtain
(modulo the relaxing of the closed requirement) what Gavin Bierman
called a ``new Seely'' category.

Next, we return to the main theme of the paper, and formulate the notion of tensorial
representation in a Cartesian storage category.  We then show that the linear
maps of a Cartesian storage category, which has persistent tensorial representation,
always form a tensor storage category.  Conversely, the coKleisli category of a
tensor storage category is a Cartesian storage category with tensorial representation.
Thus, tensor storage categories, in which the comonad is exact, correspond {\em precisely\/} to
the linear maps of Cartesian storage categories which have persistent tensorial representation.


\subsection{Coalgebra modalities}
\label{coalgebra-modality}

Let \X be a symmetric monoidal category with tensor
$(\ox,a_\ox,c_\ox, u^L_\ox,u^R_\ox)$ which has
products and a comonad $(S,\delta,\epsilon)$.

\begin{definition} \label{storage-trans}
A {\bf storage transformation} is a symmetric comonoidal transformation
$s: S \to S$ from  $\X$, regarded as a symmetric monoidal category with respect
to the Cartesian product, to $\X$, regarded as a symmetric monoidal category with
respect to the tensor, for which $\delta$ is a comonoidal transformation.
\end{definition}

Thus, a storage transformation is a natural transformation
$s_2:S(X \x Y) \to S(X) \ox S(Y)$ and a map $s_1: S(1) \to \top$ satisfying:
$$\xymatrix{S((X \x Y) \x Z) \ar[d]_{S(a_\x)} \ar[r]^{s_2} & S(X \x Y)
            \ox S(Z) \ar[r]^{s_2 \ox 1~~~}
    & (S(X) \ox S(Y)) \ox S(Z) \ar[d]^{a_\ox} \\
            S(X \x (Y \x Z)) \ar[r]_{s_2} & S(X) \ox S(Y \x Z)
            \ar[r]_{1\ox s_2~~~}
    & S(X) \ox (S(Y) \ox S(Z)) }$$

$$\xymatrix{S(1 \x X) \ar[d]_{S(\pi_1)} \ar[r]^{s_2~~} & S(1) \ox S(X) \ar[d]^{s_0 \ox 1} \\
            S(X) & \top \ox S(X) \ar[l]^{u^L_\ox}} ~~~~~~
  \xymatrix{S(X \x 1) \ar[d]_{S(\pi_0)} \ar[r]^{s_2~~} & S(X) \ox S(1) \ar[d]^{1 \ox s_0} \\
            S(X) & S(X) \ox \top \ar[l]^{u^R_\ox}}$$
$$\xymatrix{S(X \x Y) \ar[d]_{S(c_\x)}\ar[r]^{s_2~~} & S(X) \ox S(Y) \ar[d]^{c_\ox} \\
            S(Y \x X) \ar[r]_{s_2~~} & S(Y) \ox S(X) }$$

In general, a natural transformation is a comonoidal transformation if it respects the comonoidal structure
in the sense that the following diagrams  commute:
$$\xymatrix{F(X \x Y) \ar[d]_{\sigma_2^F} \ar[r]^{\alpha} & G(X \x Y) \ar[d]^{\sigma_2^G} \\
            F(X) \ox F(Y) \ar[r]_{\alpha \ox \alpha} & G(X) \ox G(Y)}  ~~~~~
  \xymatrix{F(1)\ar[dr]_{\sigma_0^F} \ar[rr]^\alpha & & G(1) \ar[dl]^{\sigma_0^G} \\ & \top}$$
A comonad is comonoidal if all its transformations are.  However, for the comonad $(S,\delta,\epsilon)$
it makes no sense to insist that $\epsilon$ is comonoidal as the identity functor is not comonoidal (from $\X$ with products to $\X$ with tensor).
However, it does make sense to require $\delta$ is comonoidal.  Recall first that, with respect
to the product, every functor is canonically comonoidal with:
$$\sigma_2^\x = \< S(\pi_0),S(\pi_1) \>: S(X \x Y) \to S(X) \x S(Y)
           ~~\mbox{and}~~ \sigma_0^\x = \<\>: S(1) \to 1.$$
This allows us to express the requirement that $\delta$ be comonoidal as follows:
$$\xymatrix{S(X \x Y) \ar[dd]_{s^2} \ar[r]^{\delta} & S(S(X \x Y)) \ar[d]^{\sigma_2^\x} \\
            & S(S(X) \x S(Y)) \ar[d]^{s_2} \\
           S(X) \ox S(Y) \ar[r]_{\delta \ox \delta~~~~} & S(S(X)) \ox S(S(Y))}
~~~~ \xymatrix{S(1) \ar[ddr]_{s_0} \ar[r]^{\delta~} & S(S(1)) \ar[d]^{S(\sigma_0^\x)} \\
                & S(1) \ar[d]^{s_0} \\
               & \top}$$

\begin{definition} \label{coalg-modality}
A symmetric monoidal category \X with products has a {\bf commutative coalgebra modality} in
case there is a comonad $(S,\delta,\epsilon)$ such that each $S(X)$ is naturally a cocommutative
comonoid $(S(X),\Delta,e)$ and $\delta$ is a homomorphism of these monoids:
$$\xymatrix{S(X) \ar[d]_{\Delta} \ar[r]^{\delta} & S(S(X)) \ar[d]^{\Delta} \\
            S(X) \ox S(X) \ar[r]_{\delta \ox \delta~~~~} & S(S(X)) \ox S(S(X))}
~~~~ \xymatrix{S(X) \ar[dr]_e \ar[rr]^{\delta~} & & S(S(X)) \ar[dl]^{e} \\ & \top}$$
\end{definition}

\noindent
We now observe that:

\begin{proposition}
For a symmetric monoidal category with products, having a comonad with a symmetric
storage transformation is equivalent to having a cocommutative coalgebra modality.
\end{proposition}

\proof
\begin{description}
\item[$(\Rightarrow)$]
If one has a storage transformation then one can define natural
transformations $\Delta, e$ as
$\Delta = S(\Delta_\x) s_2: S(X) \to S(X) \ox S(X)$ and $e = S(\<\>) s_0: S(X) \to \top$.
As (symmetric) comonoidal functors preserve (commutative) comonoids these do define comonoids.
Further, since $\delta$ is comonoidal as a transformation it becomes a homomorphism of the induced
comonoids.  This means that we have a (cocommutative) coalgebra modality.
\item[$(\Leftarrow)$]
Conversely given a (cocommutative) coalgebra modality on a (symmetric) monoidal category we
may define $s_2 = \Delta (S(\pi_0) \ox S(\pi_1)): S(X \x Y) \to S(X) \ox S(Y)$ and
$s_0 = e: S(1) \to \top$.  $s_2$ is clearly a natural transformation and using the fact that
$\Delta$ is coassociative it is easily seen to satisfy the first comonoidal requirement.  For the
last we have:
$$s_2(1 \ox s_0) u^R_\ox = \Delta(S(\pi_0) \ox S(\pi_1))(1 \ox e)u^R_\ox
  = \Delta(1) \ox S(\pi_1))u^R_\ox S(\pi_0) = S(\pi_0).$$
Thus these maps do provide comonoidal structure for $S$.  It remains to show that $\delta$ is
a comonoidal transformation:
\begin{eqnarray*}
\delta S(\sigma_2^\x) s_2 & = & \delta S(\sigma_2^\x) \Delta (S(\pi_0) \ox S(\pi_1))\\
  & = & \delta \Delta (S(\sigma_2^\x) \ox S(\sigma_2^\x)) (S(\pi_0) \ox S(\pi_1))  \\
  & = & \delta \Delta (S(\sigma_2^\x \pi_0) \ox S(\sigma_2^\x \pi_1))  \\
  & = & \delta \Delta (S(S(\pi_0)) \ox S(S(\pi_)))  \\
  & = & \Delta (\delta \ox \delta)(S(S(\pi_0)) \ox S(S(\pi_)))  \\
  & = & \Delta (S(\pi_0) \ox S(\pi_)) (\delta \ox \delta) \\
  & = & s_2 (\delta \ox \delta)
\end{eqnarray*}
\end{description}
\endproof

This not only provides an alternative way to describe a commutative coalgebra modality but also
leads us into the definition of a tensor storage category:

\begin{definition}
A {\bf tensor storage category} is a symmetric monoidal category with a
comonad $(S,\delta,\epsilon)$ which has a storage transformation which
is an isomorphism (\/{\em i.e.} both $s_2$ and $s_1$ are isomorphisms).
\end{definition}

Thinking of tensor storage categories as ``Seely categories'', one
notices that this is not the usual definition; in \cite{Mellies}
this definition (essentially) is given as a theorem which relates this
description to the more standard definition given by Bierman and adapted
by Schalk.

First, we shall compare tensor storage categories to models of the
multiplicative exponential fragment of linear logic (which we shall {\em
not\/} assume includes being closed).  We shall simply call these linear
categories as they are Bierman's linear categories without the
requirement of closure.  Thus, a {\bf linear category} is a symmetric
monoidal category $\X$ with a coalgebra modality on a symmetric monoidal
comonad $(S, \epsilon,\delta,m_\ox,m_\top)$ such that $\Delta$ and $e$
are both monoidal transformations and coalgebra morphisms.

The requirements of being a linear category are worth expanding.
The requirement that $e: S(A) \to \top$ be a monoidal transformation
amounts to two coherence diagrams:
$$\xymatrix{\top \ar@{=}[dr] \ar[r]^{m_\top~} & S(\top) \ar[d]^e \\ & \top}
~~~~ \xymatrix{S(A) \ox S(B) \ar[d]_{e \ox e} \ar[r]^{~m_\ox} & S(A \ox B) \ar[d]^e \\ \top \ox \top \ar[r]_u & \top}$$
The requirement the $\Delta: S(A) \to S(A) \ox S(A)$ is a monoidal
transformation amounts to two coherence requirements:
$$\xymatrix{\top \ar[d]_u \ar[r]^{m_\top}& S(\top) \ar[d]^{\Delta} \\
               \top \ox \top \ar[r]_{m_\top \ox m_\top~~~} & S(\top) \ox S(\top)}
~~~~ \xymatrix{S(A) \ox S(B) \ar[d]_{\Delta \ox \Delta} \ar[rr]^{m_\ox} & & S(A \ox B) \ar[dd]^{\Delta} \\
                S(A) \ox S(A) \ox S(B) \ox S(B) \ar[d]_{\ex_\ox} \\
                S(A) \ox S(B) \ox S(A) \ox S(B) \ar[rr]_{~~~m_\ox \ox m_\ox} & & S(A \ox B)\ox S(A \ox B)}$$
Requiring that $e: S(A) \to \top$ forms a coalgebra morphism amounts to
the requirement that
$$\xymatrix{S(A) \ar[r]^{\delta} \ar[d]_e &   S^2(A) \ar[d]^{S(e)} \\ \top \ar[r]_{m_\top} & S(\top)}$$
commutes.  Finally requiring that $\Delta$ is a coalgebra morphism amounts to:
$$\xymatrix{S(A) \ar[d]_{\Delta} \ar[rr]^{\delta} & & S^2(A) \ar[d]^{S(\Delta)} \\
            S(A)\ox S(A) \ar[r]_{\delta \ox \delta~~} & S^2(A)\ox S^2(A) \ar[r]_{m_\ox~} & S(S(A) \ox S(A))}$$

In his definition, Bierman had another requirement,
namely that whenever $f: S(A) \to S(B)$ is a coalgebra morphism it must
also be a comonoid morphism.  We have amalgamated this into the
definition of a coalgebra modality:

\begin{lemma} In a linear category there is the following implication of commutative diagrams
$$\infer{\xymatrix{S(A) \ar[d]_{\Delta} \ar[r]^f & S(B) \ar[d]^{\Delta} \\
                   S(A) \ox S(A) \ar[r]_{f \ox f} & S(B) \ox S(B)}
        }{\xymatrix{S(A) \ar[d]_{\delta} \ar[r]^f & S(B) \ar[d]^{\delta} \\
                   S^2(A) \ar[r]_{S(f)} & S^2(B)}
        }$$
\end{lemma}

\proof
We use the fact that $\delta$ is a morphism of comonoids and is a
section (with retraction $\epsilon$).
This means the lower diagram commutes if and only if $f \Delta (\delta \ox \delta) = \Delta (f \ox f) (\delta \ox \delta)$
but for this we have:
\begin{eqnarray*}
f \Delta (\delta \ox \delta)
& = & f \delta \Delta = \delta S(f) \Delta = \delta \Delta (S(f) \ox S(f)) \\
& = & \Delta (\delta \ox \delta) (S(f) \ox S(f)) = \Delta ((\delta S(f)) \ox (\delta (S(f)) = \Delta ((f \delta) \ox (f \delta)) \\
& = & \Delta (f \ox f) (\delta \ox \delta).
\end{eqnarray*}
\endproof

Note that this immediately means that if $f$ is a coalgebra morphism, it
is also then a comonoid morphism, and so Bierman's original definition
of a linear category corresponds to ours (modulo the absence of
closedness).

\begin{theorem}
A tensor storage category is a linear category and conversely a linear category with products is
a tensor storage category.
\end{theorem}

\proof
We rely on a combination of \cite{Schalk} and \cite{Mellies} to provide the
proof.  In particular, the monoidal structure of the monad is given by.
\begin{eqnarray*}
\top \to^{m_\top} S(\top) & = &\top \to^{s_1^{-1}} S(1) \to^\delta S^2(1) \to^{S(s_1)} S(\top) \\
S(A) \ox S(B) \to^{m_\ox} S(A \ox B)
& = & S(A) \ox S(B) \to^{s_2^{-1}} S(A \x B) \to^\delta S^2(A \x B) \\
& & \to^{S(s_2)} S(S(A) \ox S(B)) \to^{\epsilon \ox \epsilon} S(A \ox B)
\end{eqnarray*}
The converse requires checking that in a linear category with products
provides a storage isomorphism.  The fact that there is a coalgebra
modality is part of the data above.  What remains is to prove that the
induced storage transformation is an isomorphism and this is in the
literature above.
\endproof

\subsection{Tensor representation}

We now return to the main thread of the paper and consider the notion of tensorial representation for a Cartesian storage category.  Our objective is to show
that a Cartesian storage category with tensorial representation is precisely the coKleisli category of a tensor storage category.  To achieve this, however,
we must first take a detour to develop the notion of tensor representation.  Recall that a basic intuition for a tensor product is that it should represent bilinear
maps: clearly this intuition can be expressed in any category with a system of linear maps:

\begin{definition}~
\begin{enumerate}[(i)]
\item
In any Cartesian category a system of linear maps has {\bf tensorial representation} in case for each $X$ and $Y$ there is an object $X \ox Y$ and
a bilinear map $\varphi_\ox: X \x Y \to X \ox Y$ such that for every bilinear map  $g: X \x Y \to Z$ in $\X$
there is a unique linear map in $\X$ making the following diagram commute:
$$\xymatrix{X \x Y \ar[d]_{\varphi_\ox} \ar[rr]^g & & Z \\ X \ox Y\ar@{..>}[rru]_{g_{\ox[A]}} }$$
\item
In any Cartesian category a system of linear maps has {\bf strong tensorial representation} in case for each $X$ and $Y$ there is an object $X \ox Y$ and
a bilinear map $\varphi_\ox: X \x Y \to X \ox Y$ such that for every bilinear map  $g: X \x Y \to Z$ in $\X[A]$ there is a unique linear map in $\X[A]$ making
the above diagram commute.  Note that this means in $\X$ we have the diagram:
$$\xymatrix{A \x X \x Y \ar[d]_{1 \x \varphi_\ox} \ar[rr]^g & & Z \\ A\x (X \ox Y) \ar@{..>}[rru]_{g_{\ox[A]}} }$$
\item
A system of linear maps is {\bf unit representable} in case there is a linear map $\varphi_\top: 1 \to \top$
such that in $\X$ for each point $p: 1 \to Z$ there is a unique linear point $p_\top: \top \to Z$  making
$$\xymatrix{1 \ar[d]_{\varphi_\top } \ar[rr]^{p} & & Y \\
            \top \ar@{..>}[rru]_{p^{\top}}}$$
commute.
\item
A system of linear maps is {\bf strongly unit representable} in case it is representable in each simple slice.  This means there
is a unique $p^{\top[A]}$ making
$$\xymatrix{A \x 1 \ar[d]_{1 \x \varphi_\top} \ar[rr]^{p} & & Y \\
            A \x \top \ar@{..>}[rru]_{p^{\top[A]}}}$$
commute.
\item
A strong tensor representation is {\bf persistent} in case, in {(ii)} above, whenever $A = A_1 \x B \x A_2$ and the map $g$ is
linear in $B$, then $g^{\ox[A]}$ is linear in $B$.  Similarly, for a strong tensor unit representation to be persistent requires, setting $A = A_1 \x B \x A_2$  in
{(iv)}, that whenever $p$ is linear in $B$ then $p^{\top[A]}$ is also linear in $B$.
\end{enumerate}
\end{definition}

As before, the basic form of tensorial representation assumes only that it holds in the original
category: to be a strong tensorial representation requires the property must also hold in every simple slice.
Thus, as for classification, there is a progression of
notions: tensor representation, strong tensor representation, and
persistent strong tensor representation, each of which demands more than the last.  This also
applies to the unit representation; however, as we shall shortly
discover (see Lemma \ref{lem:basic_tensor_rep}), every storage category
already has a persistent strong unit representation. Thus, we shall
often talk of {\em tensorial representation} when we mean both tensor and
unit representation.

We first observe that persistence produces a simultaneous universal property:

\begin{lemma} If $f: A \x X \x Y \x Z \to W$ is linear in its last two
arguments, then there are unique maps linear in their last arguments, $f_1: A \x X \ox (Y \ox Z) \to W$ and
$f_2: A \x (X \ox Y) \ox Z \to W$ such that $(1 \x ((1 \x \varphi)\varphi)) f_1 = f = (1 \x ((\varphi\x 1)\varphi)) f_2$.
\end{lemma}

This is useful in the proof of:

\begin{proposition}
If $\X$ has a system of linear maps for which there is a persistent strong
tensorial (and unit) representation, then $\ox$ is a symmetric tensor
product with unit $\top$ on the subcategory of linear maps, ${\cal L}[]$.
\end{proposition}

\proof
When we have a persistent representation we can define the required isomorphisms for a tensor product:
\begin{enumerate}
\item The associativity isomorphism:
$$\xymatrix{A \x B\x C \ar[d]_{1 \x \varphi_\ox} \ar[rr]^{\varphi_\ox \x 1} &
                & (A \ox B) \x C \ar[rr]^{\varphi_\ox} & & (A \ox B) \ox C \\
            A \x (B \ox C) \ar[d]_{\varphi_\ox} \ar@{..>}[urrrr] \\
            A \ox (B \ox C) \ar@{..>}[uurrrr]_{a_\ox}}$$
where the key is to observe that $(1 \x \varphi_\ox)\varphi_\ox: A \x B\x C \to (A \ox B) \ox C$ is trilinear so that
the two extensions indicated are given by the universal property.
\item The symmetry isomorphism:
$$\xymatrix{A \x B \ar[d]_{\varphi_\ox} \ar[r]^{c_\x} & B \x A \ar[r]^{\varphi_\ox} & B \ox A \\
            A \ox B \ar@{..>}[rru]_{c_\ox}}$$
\item The unit isomorphisms:
$$\xymatrix{A \x 1 \ar[d]_{1 \x \varphi_\top} \ar[rr]^{\pi_1} & & A \\
            A \x \top \ar[d]_{\varphi_\ox}\ar@{..>}[rru]^{(\pi_0)^{[A]}_\top}_{\!\!\!=\,\pi_0} \\
            A \ox \top \ar@{..>}[rruu]_{u^R_\ox}}$$
Note that $(\pi_0)^{[A]}_\top = \pi_0$ since both fit here and this is bilinear whence one can define $u^R_\ox$.
It is not obvious that $u^R_\ox$ is an isomorphism, however, the diagram itself suggests that its inverse is
$$(u^R_\ox)^{-1} = \pi_1^{-1}(1 \x \varphi_\top)\varphi_\ox: A \ox \top \to A$$
and it is easily checked that this works.

We can now obtain the unit elimination on the left using the symmetry map $u^L_\ox = c_\ox u^R_\ox$.
\end{enumerate}
The coherences now follow directly from the fact that the product is a symmetric tensor and the multi-universal
property of the representation.
For example, we need $(f+g) \ox h = f \ox h + g \ox h$, which follows
from the fact that $(f+g) \ox h$ is determined by $\varphi_\ox ((f+g) \ox
h)$:

\begin{eqnarray*}
\varphi_\ox ((f+g) \ox h)
            &=_{\rlap{\scr[defn]}}\phantom{\mbox{\scr[$\varphi$ bilinear]}}& ((f+g) \x  h) \varphi_\ox \\
            &=_{\rlap{\scr[$\varphi_\ox$ bilinear]}}\phantom{\mbox{\scr[$\varphi$ bilinear]}}& (f \x h) \varphi_\ox + (g \x h) \varphi_\ox \\
            &=_{\rlap{\scr[defn]}}\phantom{\mbox{\scr[$\varphi$ bilinear]}}& \varphi_\ox (f \ox h) + \varphi_\ox (g \ox h) \\
            &=_{\rlap{\scr[left additive]}}\phantom{\mbox{\scr[$\varphi$ bilinear]}}& \varphi_\ox ((f \ox h) + (g \ox h))
\end{eqnarray*}

\endproof

When a Cartesian storage category has strong tensor representation, then this representation is automatically persistent.
 To establish this we start by observing that a Cartesian storage category already has a fair amount of tensor representation.

\begin{lemma}\label{lem:basic_tensor_rep}
In any Cartesian storage category
\begin{enumerate}[{\em (i)}]
\item $\varphi: 1 \to S(1)$ gives persistent unit tensorial representation;
\item $m_\x: S(A) \x S(B) \to S(A \x B)$ gives persistent tensor representation for the objects $S(A)$ and $S(B)$.
\end{enumerate}
\end{lemma}

\proof
We shall focus on the binary tensorial representation: first note that $m_\x$ is bilinear
as $m_\x = \theta S(\theta') \epsilon = \theta' S(\theta)  \epsilon$ where the last two maps of each
expansion are linear while $\theta$ is linear in its second argument while $\theta'$ is linear
in its first.

The universal property for an arbitrary bilinear $h$ is given by the following diagram, valid in any
slice $\X[X]$:
$$\xymatrix{A \x B \ar[drr]_{\varphi} \ar[rr]^{\varphi \x \varphi}
                  & & S(A) \x S(B) \ar[d]^{m_\x} \ar[rr]^{~~~~~~h=m((\varphi \x \varphi)h)^\sharp} & & Z \\
            & & S(A \x B) \ar@{..>}[rru]_{((\varphi \x \varphi)h)^\sharp} }$$
where $h$ must be the solution to the simultaneous classification. But this makes the
righthand triangle commute by uniqueness of the universal property.

Furthermore this is a persistent representation as the classification is persistent.
\endproof

Notice that we have shown that the Seely isomorphisms $s_0: S(1) \to \top$ and
$s_2: S(A \x B) \to S(A) \ox S(B)$ are present. These isomorphisms, which are so central to the structure
of linear logic, play an important role in what follows.

\bigskip

The observation above means that a storage category always has (persistent and strong) representation of the tensor unit.
Thus, we now focus on understanding tensorial representation.   A key observation is:

\begin{lemma} \label{bilinear-coequalizes}
In a Cartesian storage category $\X$, $f: A \x B \to C$ is
bilinear if and only if its classifying linear map $f^\sharp: S(A \x B) \to C$ coequalizes
the linear maps $(S(\epsilon)\ox S(\epsilon)) s^{-1}_2$ and $(\epsilon\ox\epsilon) s^{-1}_2$.
\end{lemma}

We shall use tensor notation even though we are not assuming that we have tensorial representation: this is justified by the fact that, for these particular objects,
we are {\em always\/} guaranteed representation by Lemma \ref{lem:basic_tensor_rep}.   The translation of  these maps, $S(\epsilon)\ox S(\epsilon)$ and
$\epsilon_{S(A)} \ox \epsilon_{S(A)}$, back into tensor-free notation uses the commuting diagrams below:
$$\xymatrix{S(S(A) \x S(B)) \ar[d]_{s_2}  \ar[rr]^{S(\epsilon \x \epsilon)} & & S(A \x B) \ar[d]^{s_2} \\
                    S^2(A) \ox S^2(B) \ar[rr]_{S(\epsilon)\ox S(\epsilon)} & & S(A) \ox S(B)  }$$
$$\xymatrix{S(S(A) \x S(B))  \ar[d]_{s_2}  \ar[rr]^{\epsilon} & & S(A) \x S(B) \ar[rrd]_{\varphi_\ox} \ar[rr]^{m_\x} & & S(A \x B) \ar[d]^{s_2} \\
                    S^2(A) \ox S^2(B) \ar[rrrr]_{\epsilon\ox \epsilon} & && & S(A) \ox S(B)  }$$

\proof
We shall use the characterization of bilinear maps in Proposition
\ref{prop:linmaps} to show that coequalizing these maps is equivalent to bilinearity.

Assume that $(S(\epsilon)\ox S(\epsilon)) s_2^{-1} f^\sharp = (\epsilon\ox\epsilon) s_2^{-1} f^\sharp$ so that
$$(\varphi \x \varphi) \varphi_\ox (S(\epsilon)\ox S(\epsilon)) s_2^{-1} f^\sharp = (\varphi \x \varphi) \varphi_\ox (\epsilon\ox\epsilon) s_2^{-1} f^\sharp$$
But then we have:
\begin{eqnarray*}
\lefteqn{(\varphi \x \varphi) \varphi_\ox (S(\epsilon)\ox S(\epsilon)) s_2^{-1} f^\sharp} \\
& = & (\varphi \x \varphi) (S(\epsilon)\x S(\epsilon)) \varphi_\ox s_2^{-1} f^\sharp \\
& = & (\epsilon \x \epsilon) (\varphi \x \varphi) m_\x f^\sharp \\
& = &  (\epsilon \x \epsilon) \varphi f^\sharp \\
& = &  (\epsilon \x \epsilon) f \\
\lefteqn{(\varphi \x \varphi) \varphi_\ox (\epsilon\ox\epsilon) s_2^{-1} f^\sharp} \\
& = & (\varphi \x \varphi) (\epsilon\x\epsilon) \varphi_\ox s_2^{-1} f^\sharp \\
& = & \varphi_\ox s_2^{-1} f^\sharp \\
& = & m_\x S(f)\epsilon
\end{eqnarray*}
So this condition implies bilinearity.   Conversely, by reversing the
argument, assuming bilinearity gives---using classification and tensorial representation---the equality of these maps.
\endproof

We now have:

\begin{proposition} \label{basic-tensor-rep}
$\X$ has (basic) tensor representation at $A$ and $B$ if and only if
$$S^2(A)\ox S^2(B) \Two^{S(\epsilon)\ox S(\epsilon)}_{\epsilon\ox\epsilon} S(A)\ox S(B) \to^{\epsilon \ox \epsilon} A \ox B$$
is a coequalizer in the subcategory of linear maps.
\end{proposition}

\begin{proof}
Suppose $z$ equalizes $S(\epsilon) \ox S(\epsilon)$ and $\epsilon \ox \epsilon$ then by the previous lemma $z' = (\varphi \x \varphi) \varphi_\ox z$ is bilinear
and this determines a unique linear map $z'_\ox$ with $z'= \varphi_\ox z'_\ox$.
$$\xymatrix{ S^2(A) \x S^2(B) \ar[d]_{\varphi_\ox} \ar@<1ex>[rr]^{S(\epsilon) \x S(\epsilon)}  \ar@<-1ex>[rr]_{\epsilon \x \epsilon}
                                       && S(A) \x S(B)  \ar[d]^{\varphi_\ox} \ar[rr]^{\epsilon \x \epsilon} && A \x B \ar[d]^{\varphi_\ox} \ar@{..>}@/^1pc/[ddrr]^{z'}\\
                     S^2(A) \ox S^2(B) \ar@<1ex>[rr]^{S(\epsilon) \ox S(\epsilon)}  \ar@<-1ex>[rr]_{\epsilon \ox \epsilon}
                                       && S(A) \ox S(B) \ar[rr]_{\epsilon \ox \epsilon} \ar@/_1pc/[drrrr]_z&& A \ox B \ar@{..>}[drr]_{z'_\ox}\\
                                       && && && Z}$$
We claim $z'_\ox$ is the unique comparison map making the fork a coequalizer in the linear map category.  To show this we need  $(\epsilon \ox \epsilon) z'_\ox = z$ and
we must show that any other linear map $k$ with $(\epsilon \ox \epsilon) k = z$ has $\varphi_\ox k = z'$.  For the first of these we note that $z$ is determined by
representing $(\varphi \x \varphi)\varphi_\ox z = z'$ but
$$(\varphi \x \varphi)\varphi_\ox (\epsilon \ox \epsilon) z'_\ox = (\varphi \x \varphi) (\epsilon \x \epsilon) \varphi_\ox z'_\ox = \varphi_\ox z'_ox = z'$$
so the two maps are equal.  For the second, suppose we have such a $k$ then
$$\varphi_\ox k = (\varphi \x \varphi)(\epsilon \x \epsilon) \varphi_\ox k = (\varphi \x \varphi) \varphi_\ox (\epsilon \ox \epsilon) k = (\varphi \x \varphi) \varphi_\ox z = z'.$$

For the converse assume that the linear map category has this fork a coequalizer and suppose $f:A \x B \to C$ is bilinear.  Returning to the diagram above, by Lemma \ref{bilinear-coequalizes} we may set $z'=f$ and $z= s^{-1}_2f^\sharp$.  Noting that the top fork is an absolute coequalizer gives the map $\varphi_\ox: A \x B \to A \ox B$ which will then represent bilinear maps.
\end{proof}

Considering what happens in an arbitrary simple slice gives:

\begin{corollary} \label{strong-tensor-rep}
A storage category has strong tensorial representation if and only if forks of the form:
$$X \x S^2(A) \ox S^2(B) \Two^{1 \x S(\epsilon) \ox S(\epsilon)}_{1 \x \epsilon \ox \epsilon} X \x S(A) \ox S(B) \to^{1 \x \epsilon \ox \epsilon} X \x A \ox B$$
are coequalizers for the linear subcategories ${\cal L}[X]$
\end{corollary}

There are a number of different reasons why these forks might be coequalizers.   It is often the case that these forks will also be coequalizers in the whole storage category.  When this is so, one may decompose the presence of these coequalizers into the presence of the basic coequalizer (of Proposition \ref{basic-tensor-rep}) and the fact they must be preserved by the functor $X \x \_$.  Because these coequalizers are clearly reflexive, this latter condition is delivered whenever the product functor preserves, more generally, reflexive coequalization.  Of course, when the storage category is Cartesian closed, the product functor will preserve {\em all\/} colimits and so in particular these coequalizers.   In fact, the case which will be of primary interest to us here, as shall be discussed below, is when these coequalizers are absolute and so are automatically preserved by the product functors and, furthermore, are present under the very mild requirement that linear idempotents split.

We have established that these coequalizers must be present in any
storage category with strong tensor representation.    The final
ingredient we need is persistence: fortunately this is guaranteed once
strong tensor representation is assumed.

\begin{proposition}
In a Cartesian storage category with a strong tensor representation, the representation is necessarily persistent: that is,
given a map $h:X_0\x C\x X_1\x A\x B\to Y$, which is linear in $C$, $A$, and $B$
\[\xymatrix{
X_0\x C\x X_1\x A\x B  \ar[r]^{~~~~~~~~~~~~~h} \ar[d]_{1\x\varphi_{\ox}} & Y \\
X_0\x C\x X_1\x (A\ox B) \ar[ur]_{h^\ox}
}\]
then its linear lifting $h_\ox:X_0\x C\x X_1\x (A\ox B) \to Y$ is linear
in $C$.
\end{proposition}

\begin{proof}
By using Lemma \ref{simple-slice-context} we may simplify what needs to be proven: namely, given $f: A \x B \x C \to Y$ which is linear in all its
arguments then $f^\ox: A \ox B \x C \to Y$ is linear in $C$.   This latter requires that we show that
$$\xymatrix{A \ox B \x S(C) \ar[d]_{1 \x \epsilon} \ar[r]^{\varphi \x 1} & S(A \ox B) \x S(C) \ar[r]^{m_\x} & S(A \ox B \x C) \ar[r]^{f^\ox} & S(Y) \ar[d]^{\epsilon} \\
                    A \ox B \x C \ar[rrr]_{f^\ox} &&& Y}$$
commutes.   Because we have strong tensor representation we know $(\epsilon \ox \epsilon) \x 1 :S(A) \ox S(B) \x C \to A \ox B \x C$ is epic.  Thus, we may preface this
square with the map $(s^2 \x 1) ((\epsilon \ox \epsilon) \x 1)$ to test commutativity.   Preliminary to this calculation note that whenever $f$ is bilinear the following diagram
$$\xymatrix{S(A) \x S(B)  \ar[dd]^{\epsilon \x \epsilon} \ar[r]_{m_\x} \ar[rd]_{\varphi_\ox}
                     & S(A \x B) \ar[d]^{s_2} \ar[r]_{S(f)} & S(Y) \ar[dd]^{\epsilon} \\
                     & S(A) \ox S(B) \ar@{..>}[dr] \ar[d]_{\epsilon \ox \epsilon} \\
                     A \x B \ar@/_2pc/[rr]_{f} \ar[r]^{\varphi_\ox} & A \ox B \ar[r]^{f^\ox} & Y}$$
commutes as the dotted arrow is the unique linear extension of the bilinear map $m_\x S(f) \epsilon = (\epsilon \x \epsilon) f$.
Putting the right square in the simple slice over $C$ gives the following commuting square:
$$\xymatrix{S(A \x B) \x C \ar[d]_{s_2 \x 1} \ar[r]^{(1 \x \varphi) m_\x} & S(A \x B \x C) \ar[r]^{~~~~S(f)} & S(Y) \ar[dd]^{\epsilon} \\
                    S(A) \ox S(B) \x C \ar[d]_{\epsilon \ox \epsilon \x 1} \\
                    A \ox B \x C \ar[rr]_{f_\ox} && Y}$$
We now have the calculation:
\begin{eqnarray*}
(s_2(\epsilon \ox \epsilon) \x 1)(1 \x \epsilon) f^\ox
& = & (1 \x \epsilon)(s_2(\epsilon \ox \epsilon) \x 1)f^\ox \\
& = &  (1 \x \epsilon)(1 \x \varphi)m_\x S(f) \epsilon \\
& = & m_\x S(f) \epsilon \\
(s_2(\epsilon \ox \epsilon) \x 1)(\varphi \x 1)m_\x S(f^\ox) \epsilon
&= & (\varphi \x 1)m_\x S((s_2(\epsilon \ox \epsilon) \x 1)f^\ox) \epsilon \\
& = & (\varphi \x 1)m_\x S((1 \x \varphi)m_\x S(f) \epsilon)\epsilon \\
& = & (\varphi \x 1)m_\x S((1 \x \varphi)m_\x)\epsilon S(f) \epsilon \\
& = & m_\x S(f) \epsilon
\end{eqnarray*}
where the last step crucially uses the commutativity of the monad.
\end{proof}

When one has tensor representation, if the tensor preserves
the coequalizers which witness the exactness of the modality $S$, that is coequalizers of the form
$$S^2(A) \Two^{S(\epsilon)}_{\epsilon} S(A) \to^\epsilon A$$
then the coequalizer above can be further analyzed {\em via} the following parallel coequalizer diagram:
\[
\xymatrix{
S^2(A) \ox S^2(B) \ar@<1ex>@{<-}@/^1.5pc/[rr]^{\delta\ox1} \ar@<1ex>[rr]^{\epsilon \ox 1} \ar@<-1ex>[rr]_{S(\epsilon) \ox 1}
                  \ar@<-1ex>@{<-}@/_3pc/[dd]_{1\ox\delta} \ar@<1ex>[dd]^{1\ox\epsilon} \ar@<-1ex>[dd]_{1\ox S(\epsilon)}
                  \ar@<1ex>[ddrr]^{\epsilon\ox\epsilon} \ar@<-1ex>[ddrr]_{S(\epsilon)\ox S(\epsilon)}
       && S(A)\ox S^2(B) \ar[rr]^{\epsilon\ox1}
                  \ar@<1ex>[dd]^{1\ox\epsilon} \ar@<-1ex>[dd]_{1\ox S(\epsilon)}
       && A\ox S^2(B)  \ar@<1ex>[dd]^{1\ox\epsilon} \ar@<-1ex>[dd]_{1\ox S(\epsilon)}
\\  &&&& \\
S^2(A)\ox S(B) \ar@<1ex>[rr]^{\epsilon \ox 1} \ar@<-1ex>[rr]_{S(\epsilon) \ox 1} \ar@<1ex>[dd]^{1\ox\epsilon}
       && S(A)\ox S(B) \ar@<1ex>[dd]^{1\ox\epsilon} \ar[rr]^{\epsilon\ox1} \ar[ddrr]^{\epsilon\ox\epsilon}
       && A\ox S(B)  \ar@<1ex>[dd]^{1\ox\epsilon}
\\  &&&& \\
S^2(A)\ox B  \ar@<1ex>[rr]^{\epsilon \ox 1} \ar@<-1ex>[rr]_{S(\epsilon) \ox 1}
      && S(A)\ox B \ar[rr]_{\epsilon\ox1}
      && A\ox B
}
\]
As proven in \cite[Lemma 1.2.11]{PTJElephant}, for example, given
horizontal and vertical reflexive coequalizers as shown above, the
diagonal is also a coequalizer.   Recall that the basic coequalizers
are actually absolute coequalizers in the storage category so they
certainly are coequalizers in the linear category (use {\bf LS.2]}) but
they will not necessarily be absolute in the linear category.  Thus,
assuming that the tensor product preserves them is a far from benign
assumption.


\subsection{Codereliction and  tensor representation}

An important source of tensorial representation in Cartesian storage categories arises from the presence of a {\bf codereliction}: this is an
unnatural transformation $\eta: A \to S(A)$ which {\em is\/} natural for linear maps such that each $\eta_A$ is
linear and splits $\epsilon$ in the sense that $\eta_A\epsilon_A = 1_A$.
Notice that $\varphi$, in general, will not be a codereliction: is natural for all maps but it is not in general
linear---in fact if it were linear then all maps would be linear and $S$ would
necessarily be equivalent to the identity functor.

As we shall shortly see, all Cartesian differential storage categories have a codereliction.  This is important because of the following
observation:

\begin{proposition}
Any Cartesian storage category in which linear idempotents split
linearly, and which has a codereliction, has persistent tensorial
representation.
\end{proposition}

Recall  that a linear idempotent $e$ splits linearly when there is a
splitting $(r,s)$ with $rs = e$ and $sr = 1$ such that both $s$ and $r$
are linear.  It follows from Lemma \ref{storage-splitting}, that both
$r$ and $s$ are linear if either is. Furthermore, it follows from
Proposition \ref{linear-idempotent-splitting} that we may always
formally split linear idempotents.

\begin{proof}
Once we have a codereliction we have a split pair
\[
\xymatrix{
S^2(A) \ox S^2(B) \ar@<1ex>@{<-}@/^1.5pc/[rr]^{\eta \ox\eta} \ar@<1ex>[rr]^{\epsilon \ox \epsilon} \ar@<-1ex>[rr]_{S(\epsilon) \ox S(\epsilon)}
       && S(A) \ox S(B)
}
\]
whose absolute coequalizer is the splitting of $\epsilon \eta \ox
\epsilon\eta$.   Thus, when linear idempotents split one has basic
tensor representation. As this coequalizer is absolute in the whole
storage category it is necessarily preserved by products and so the
tensor representation is strong, hence persistent.
\end{proof}

In particular, as we shall discover in Section \ref{diff-storage}, a Cartesian differential storage category always has a codereliction
$\eta = \< 1,0 \> D_\x[\varphi]$ which is linear (in the differential sense) and splits $\epsilon$.  Thus, in these examples it
suffices for linear idempotents (in simple slices) to have linear splittings.  This can always be formally arranged by splitting the linear idempotents.


\subsection{Tensor storage categories with an exact modality}

We now observe that in a Cartesian storage category, which has
tensorial representation, the linear maps form a tensor storage category:

\begin{proposition}
In any Cartesian storage category with tensorial representation,
the subcategory of linear maps is a tensor storage category.
\end{proposition}

\proof
We have already observed that there is a monoidal comonad present on the linear maps.  In addition,
when both a persistent classification and representation is present there is are natural isomorphisms
$s_\ox: S(X \x Y) \to S(X) \ox S(Y)$ and an isomorphism  $s_\top: \top \to S(1)$ constituting an
iso-comonoidal structure for the functor $S: \X \to \X$ from $\X$ with product to $\X$ with tensor: this is
the storage isomorphism.

This follows immediately from the fact that
$$\xymatrix{X \x Y \ar[rrd]_{\varphi} \ar[r]^{\varphi \x \varphi} & S(X) \x S(Y) \ar[r]^{\varphi_\ox}
              & S(X) \ox S(Y) \ar@{..>}[d]^{s_\ox^{-1}} \\
              & & S(X \x Y) } ~~~~\mbox{and}~~~
 \xymatrix{X \x Y \ar[rd]_{(\varphi \x \varphi)\varphi_\ox} \ar[r]^{\varphi}
              & S(X \x Y)  \ar@{..>}[d]^{s_\ox} \\ & S(X) \ox S(Y) }$$
have the same universal property.  Similarly $\top$ and $S(1)$ have the same universal property.
That this constitutes a storage transformation is straightforward to check.
\endproof

To establish the converse of this observation we need to prove that the coalgebra modality of a
tensor storage category is a comonad with a commutative force:

\begin{proposition} \label{force-for-tensor-storage}
The coKleisli category of a tensor storage category is a Cartesian
storage category with tensorial representation.
\end{proposition}

\proof
We shall show that the comonad is forceful where the force is defined as:
\begin{eqnarray*}
\lefteqn{S(A \x S(X)) \to^\psi S(A \x X)} \\
& = & S(A \x S(X)) \to^{s_2} S(A) \ox S^2(X) \to^{1 \ox \epsilon} S(A) \ox S(X)  \to^{s_2^{-1}} S(A \x X).
\end{eqnarray*}
We now have to check the six coherence diagrams of a force:
\begin{enumerate}[{\bf [Force.1]}]
\item For the associativity of force we use the fact that the Seely isomorphism is comonoidal:
\begin{eqnarray*}
\lefteqn{\delta S(\sigma_\x) S(\epsilon \x \psi) S(\psi) S(a_\x)}\\
 & = & \delta S(\sigma_\x) S(\epsilon \x 1)S(1 \x (s_2 (1 \ox \epsilon)s_2^{-1})) s_2 (1 \ox \epsilon)s_2^{-1} S(a_\x) \\
 & = & \delta S(\sigma_\x) s_2 (\epsilon \ox S(s_2 (1 \ox \epsilon)s_2^{-1})) (1 \ox \epsilon) s_2^{-1} S(a_\x) \\
 & = & \delta S(\sigma_\x) s_2 (\epsilon \ox (\epsilon s_2 (1 \ox \epsilon))) (1 \ox s_2^{-1}) s_2^{-1} S(a_\x) \\
 & = & \delta S(\sigma_\x) s_2 (\epsilon \ox (\epsilon s_2 (1 \ox \epsilon))) a_\ox (s_2^{-1} \ox 1) s_2^{-1} \\
 & = & \delta S(\sigma_\x) s_2 (\epsilon \ox (\epsilon s_2))) a_\ox (s_2^{-1} \ox \epsilon) s_2^{-1} \\
 & = & s_2 (\delta \ox \delta) (\epsilon \ox (\epsilon s_2)) a_\ox (s_2^{-1} \ox \epsilon) s_2^{-1} \\
 & = & s_2 (1 \ox s_2) a_\ox (s_2^{-1} \ox \epsilon) s_2^{-1}
  =  S(a_\x) s_2 (1 \ox \epsilon)s_2^{-1}
  =  S(a_\x) \psi
\end{eqnarray*}
\item For projection of force:
\begin{eqnarray*}
\psi S(\pi_1) & = & s_2 ( 1 \ox \epsilon) s_2^{-1} S(\pi_1) \\
              & = & s_2(1 \ox \epsilon) (e \ox 1) u^\top_L \\
              & = & s_2 (e \x 1) u^\top_L \epsilon \\
              & = & S(\pi_1) \epsilon
\end{eqnarray*}
\item For forceful naturality:
\begin{eqnarray*}
\delta S(\sigma_\x)S(1 \x (\epsilon\delta)))\psi
& = & \delta S(\sigma_\x)S(1 \x (\epsilon\delta)) s_2 (1 \ox \epsilon) s_2^{-1} \\
& = & \delta S(\sigma_\x)s_2 (1 \ox S(\epsilon\delta)) (1 \ox \epsilon) s_2^{-1} \\
& = & s_2 (\delta \ox \delta) (1 \ox S(\epsilon\delta)) (1 \ox \epsilon) s_2^{-1} \\
& = & s_2 (1 \ox \epsilon) (\delta \ox \delta)  s_2^{-1} \\
& = & s_2 (1 \ox \epsilon) s_2^{-1} \delta S(\sigma_\x) \\
& = & \psi \delta S(\sigma_\x)
\end{eqnarray*}
\item For forcefulness of counit:
\begin{eqnarray*}
\psi\psi & = & s_@ (1 \ox \epsilon) s_2^{-1}s_2(1 \ox \epsilon)s_2^{-1} \\
& = & s_2(1 \ox \epsilon\epsilon)s_2^{-1} \\
& = & s_2(1 \ox S(\epsilon)\epsilon)s_2^{-1} \\
& = & S(1 \x \epsilon) s_2 (1 \ox \epsilon) s_2^{-1} \\
& = & S(1 \x \epsilon) \psi
\end{eqnarray*}
\item For forcefulness of comultiplication:
\begin{eqnarray*}
\delta S(\sigma_\x) \psi S(\epsilon \x 1)
       & = & \delta S(\sigma_\x) s_2 (1 \ox \epsilon) s_2^{-1} S(\epsilon \x 1) \\
       & = & s_2 (\delta \ox \delta) (S(\epsilon) \ox \epsilon) s_2^{-1} = 1
\end{eqnarray*}
\item Commutativity of force:
\begin{eqnarray*}
\psi \psi' & = & s_2 (1 \ox \epsilon) s_2^{-1} s^2 (\epsilon \ox 1) s^{-1}_2 \\
           & = & s_2 (1 \ox \epsilon)(\epsilon \ox 1) s^{-1}_2 \\
           & = & s_2 (\epsilon \ox 1)(1 \ox \epsilon) s^{-1}_2 \\
           & = & s_2 (\epsilon \ox 1)s_2^{-1} s^2 ((1 \ox \epsilon) s^{-1}_2 \\
           & = & \psi'\psi
\end{eqnarray*}
\end{enumerate}

\medskip

This shows that the coKleisli category is a Cartesian storage category;
it remains to show that it has tensorial representation.  However, this
follows, since in the coKleisli category $A \x B
\to^{\varphi_\ox} A \ox B$ is given by the $\X$-map
$$S(A \x B) \to^{s_2} S(A) \ox S(B) \to^{\epsilon \ox \epsilon} A \ox B$$
That this represents the tensor is easily seen, and persistence is
automatic in a coKleisli category.
\endproof

\bigskip

The results above tell us that a tensor storage category with an exact
modality is always the category of linear maps of some Cartesian storage
category with tensorial representation.  This is because its coKleisli
category is a Cartesian storage category with tensorial representation
and one can recover, as the linear maps, the original tensor storage
category.


\subsection{Closed storage categories} \label{closedstorage}

We briefly return to the issue of the closeness of these various categories.  Starting with a tensor storage category (Seely category) it is well-known that if
the category is closed in the sense that there is an adjunction
$$\infer={ X \to A \lollipop B}{A \ox X \to B}$$
then the coKleisli category is Cartesian closed with
$$A \Rightarrow B := S(A) \lollipop B$$
as we have the following natural equivalences:
$$\infer={S(X) \to S(A) \lollipop B}{
    \infer={S(A) \ox S(X) \to B}{
    S(A \x X ) \to B}}$$
Thus, if the tensor storage category is closed then the coKleisli category, that is the Cartesian storage category, must be Cartesian closed.   We wish
to work for a converse, namely, knowing that the Cartesian storage category is Cartesian closed, can we provide some natural conditions for the linear
maps to form a monoidal closed category?  The conditions that we shall consider, in fact, do not require tensorial representation and, in fact, make the linear maps
into a closed category in the original sense of Kelly and Eilenberg \cite{KE}.

We shall say that a Cartesian storage category is {\bf closed} in case
it is Cartesian closed, the linear maps form a closed system (see
Definition \ref{closed-system-defn}), the functor $S$ is enriched over
itself, and the following equalizer (which defines the object $A \lollipop B$) exists for each $A$ and $B$:
$$A \lollipop B \to^{k_{AB}} A \Rightarrow B \Two^{S_{AB} (1 \Rightarrow \epsilon)}_{\epsilon \Rightarrow 1} S(A) \Rightarrow B$$
with $k_{AB}$ linear for each $A$ and $B$.  Here the map $S_{AB}: A \Rightarrow B \to S(A) \Rightarrow S(B)$ is given by the enrichment of the functor $S$.  Clearly we are
intending that this equalizer should provide the object of linear maps from $A$ to $B$.  That it has this property relies on the following:

\begin{proposition} \label{closed-prop} In any closed Cartesian storage category:
\begin{enumerate}[{\em (i)}]
\item
If $f: A \x X \to B$ is linear in its first argument  then there is a unique map $\tilde{f}: X \to A \lollipop B$ making
$$\xymatrix{ A \x X \ar[rr]^f \ar@{..>}[d]^{1 \x \tilde{f}}& & B \\ A \x A \lollipop B \ar[r]_{1 \x k_{AB}} & A \x A \Rightarrow B \ar[ur]_{{\sf ev}} }$$
commute.
\item Any $f$ which can be expressed as $f = (1 \x \tilde{f}) (1 \x k_{AB}) {\sf ev}$ must be linear in its first argument.
\item If $f$ is bilinear then $\tilde{f}$ is linear.
\end{enumerate}
\end{proposition}

\begin{proof}~
\begin{enumerate}[{\em (i)}]
\item
Clearly $\tilde{f} k_{AB} $ must be the unique map to the hom-object $A \Rightarrow B$, $\hat{f} = \tilde{f}k_{AB}: X \to A \Rightarrow B$, and so, as $k_{AB} $ is monic, $\tilde{f}$, must be unique, if it exists.  The difficulty is to show that $\hat{f}$ factors through this equalizer.

Recall that as $f$ is linear in its first argument we have
$$\xymatrix{S(A) \x X \ar[d]_{\epsilon \x 1} \ar[r]^{\theta'} & S(A \x X) \ar[r]^{S(f)} & S(B) \ar[d]^{\epsilon} \\
                    A \x X \ar[rr]_f & & B}$$
commutes (that is $(1 \x \epsilon) f = \theta'S(f)\epsilon$).  But then we have the following two commuting diagrams displaying the curried map for these two maps:
$$\xymatrix{S(A) \x X \ar[d]_{! \x \hat{f}} \ar[r]^{\epsilon \x 1} & A \x X  \ar[d]^{1 \x \hat{f}}\ar@/^1pc/[rrd]^f \\
                    S(A) \x A \Rightarrow B \ar[d]_{1 \x (\epsilon \Rightarrow 1)} \ar[r]_{\epsilon \x 1}  & A \x A \Rightarrow B \ar[rr]_{\sf ev} & & B \\
                    S(A) \x S(A) \Rightarrow B   \ar@/_1pc/[rrru]_{\sf ev} }$$
This shows that ${\sf curry}((1 \x \epsilon) f) = \hat{f} (\epsilon \Rightarrow 1)$.
$$\xymatrix{S(A) \x X \ar[d]_{1 \x \hat{f}} \ar[r]^{\theta'} & S(A \x X) \ar[r]^{S(f)} & S(B) \ar[ddrr]^{\epsilon} \\
                    S(A) \x A \Rightarrow B \ar[d]_{1 \x S_{AB}} \\
                    S(A) \x S(A) \Rightarrow S(B) \ar[d]_{1 \x (1 \Rightarrow \epsilon)}\ar@/_1pc/[rruu]_{\sf ev} & & & & B \\
                    S(A) \x S(A) \Rightarrow B  \ar@/_1pc/[rrrru]_{\sf ev} }$$
While this shows that ${\sf curry}(\theta' S(f) \epsilon) = \hat{f} S_{AB} (1 \Rightarrow \epsilon)$.  Thus $\hat{f}$ factors through the equalizer as desired.
\item For the converse we have:
$$\xymatrix{S(A) \x X \ar[dr]_{! \x \hat{f}} \ar[ddd]^{\epsilon \x 1} \ar[rr]^{\theta'} & & S(A \x X) \ar[rr]^{S(f)} & &  S(B) \ar[ddd]^{\epsilon} \\
                     & S(A) \x A \lollipop B \ar[rrd]_{1 \x (k_{AB}(\epsilon \Rightarrow 1))} \ar[rr]^{1 \x (k_{AB}S_{AB})~~} & & S(A) \x S(A) \Rightarrow S(B)  \ar[ru]_{\sf ev} \ar[d]^{1 \x (1 \Rightarrow \epsilon)} \\
                     & & & S(A) \x S(A) \Rightarrow B \ar[dr]^{\sf ev} \\
                     A \x X \ar[rrrr]_f  &&&& B}$$
\item  It remains to show that $\tilde{f}$ is linear when $f$ is bilinear that is $\epsilon \tilde{f} = S(\tilde{f})\epsilon$ for  which we have:
$$\xymatrix{A \x S(X) \ar[d]_{1 \x \epsilon\tilde{f}} \ar[r]^{1 \x \epsilon} & A \x X \ar[d]^f \ar[dl]_{1 \x \tilde{f}} \\
                    A \x A \lollipop B \ar[r]_{~~~~{\sf ev}_\lollipop} & B}$$
$$\xymatrix{ A \x S(X) \ar[r]^{\theta} \ar[d]_{1 \x S(\tilde{f}) }& S(A \x X) \ar[r]^{S(f)} & S(B) \ar[dd]^{\epsilon} \\
                    A \x S(A \lollipop B) \ar[r]_{\theta} \ar[d]_{1 \x \epsilon} & S(A \x A \lollipop B) \ar[ru]_{S({\sf ev}_\lollipop)} \\
                    A \x A \lollipop B \ar[rr]_{{\sf ev}_\lollipop} & & B}$$
                    Showing linearity in the second argument implies that $\tilde{f}$ is linear.
\end{enumerate}
\end{proof}

In particular, notice that this immediately means that ${\sf ev}_\lollipop := (1 \x k_{AB}) {\sf ev}$ is linear in its first argument, as ${\sf ev}_\lollipop$ defined in this manner certainly satisfies
Proposition \ref{closed-prop} {\em (ii)}.
However, as $k_{AB}$ is linear and ${\sf ev}$ is linear in its second argument---as we are assuming a closed linear system---it follows that ${\sf ev}_\lollipop$ is bilinear.
Thus, when one has tensorial representation this means that there is an induced evaluation map ${\sf ev}_\ox: A \ox (A \lollipop B) \to B$ with
$\varphi_\ox {\sf ev}_\ox = {\sf ev}_\lollipop$ which is linear.  Furthermore, by Proposition \ref{closed-prop} (iii) the curry map for $f: A \ox X \to B$ is also linear
as it is $\widetilde{\varphi_\ox f}$.  Thus we have:

\begin{corollary}
If $\X$ is a closed Cartesian storage category which has tensor representation, then the subcategory of linear maps forms a symmetric monoidal closed category.
\end{corollary}

This leaves open the converse: if one starts with a monoidal
closed tensor storage category, will its coKleisli category be a closed
Cartesian storage category? We shall now show that this is true, giving
a complete characterization for the closed case:

\begin{proposition}
If $\X$ is a tensor storage category which is monoidal closed then its coKleisli category is a closed Cartesian storage category.
\end{proposition}

\begin{proof}
We provide a sketch of the proof.  We must show three things:
\begin{enumerate}
\item We need to show that the coKleisli category is a Cartesian closed category and that the functor induced by the modality
is suitably enriched.   As above, the closed structure is given by $A \Rightarrow B := S(A) \lollipop B$.  In a Cartesian closed category
a functor is enriched whenever it is strong, so the fact that there is a natural force $\psi: S(A \x S(B)) \to S(A \x B)$ (described in
Proposition \ref{force-for-tensor-storage}) guarantees the functor is strong.
\item  We need to show that the linear maps of the coKleisli category form a closed system.   We know already they form a linear system
so we need only check that they are a {\em closed\/} linear system.
This amounts to checking the following.
\begin{enumerate}
\item The evaluation map is linear in its higher-order argument. This is
immediately true by inspection. Here is the definition of the evaluation
map in the coKleisli category:
$$S(A \x  (S(A) \lollipop B)) \to^{s_2} S(A) \ox S(S(A) \lollipop B) \to^{1 \ox \epsilon}  S(A) \ox (S(A) \lollipop B) \to^{{\sf ev}_\lollipop} B$$
The second map in the sequence ensures it is linear in the second argument.
\item Linearity is ``persistent'' over currying.  To say that a coKleisli map $f: S(A \x B \x C) \to D$ is linear in its second argument amounts to saying
that $f = s_3 (1 \ox \epsilon \ox 1) f'$.  Currying this map (with respect to the first argument) clearly maintains the linearity in the second argument.
\end{enumerate}
\item This leaves only the requirement that $A \lollipop B$ occurs as an equalizer:
$$A \lollipop B \to^{k_{AB}} A \Rightarrow B \Two^{S_{AB} (1 \Rightarrow \epsilon)}_{\epsilon \Rightarrow 1} S(A) \Rightarrow B$$
Recall that we may assume the modality is exact (if it is not, simply work in the subcategory of linear maps).   This equalizer (after considerable
unwinding) is exactly the image of
$$A \lollipop B \to^{\epsilon \lollipop 1} S(A) \lollipop B \Two^{S(\epsilon)  \lollipop 1}_{\epsilon \lollipop 1} S^2(A) \lollipop B$$
under the inclusion into the coKleisli category.  Since
$$S^2(A) \Two^{S(\epsilon)}_{\epsilon} S(A) \to^{\epsilon} A$$
under this inclusion is a coequalizer, the transpose of this is an equalizer.  This shows that $A \lollipop B$ does occur as an equalizer as required.
\end{enumerate}
\end{proof}


\section{Cartesian differential storage categories} \label{diff-storage}


Tensor differential categories were introduced in \cite{diffl}:  they consist of a symmetric monoidal category with a coalgebra modality
and a deriving transformation.   In \cite{CartDiff} the notion of a Cartesian differential category was introduced.  It was proven that an important way
in which Cartesian differential categories arise is as the coKleisli categories of tensor differential categories which satisfied an additional interchange requirement
on the deriving transformation.  This latter requirement was presented as a sufficient condition: the question of whether it was necessary was left open.  Here we partially
answer that question for an important class of tensor differential categories: those whose modality is a storage modality---in other words those which are
``Seely categories''  in the sense discussed above---and exact: we show that this condition is indeed necessary.

Not surprisingly, the strategy of the proof is to characterize the coKleisli category of a tensor differential category with a storage modality using the development
above.  The fact that the coKleisli category of a tensor differential category (satisfying the interchange condition) is necessarily a Cartesian differential category
means we should consider Cartesian storage categories which are simultaneously Cartesian differential categories.  Furthermore, we expect the linear maps
in the differential sense to be the linear maps in the storage sense.  We will then show that, under these assumptions, the linear maps form a tensor differential category
which satisfies the interchange requirement.  This then provides a {\em characterization\/}  of coKleisli categories of tensor differential categories (with a storage modality)
which form differential categories.

It is worth mentioning---not least as it caused the authors some
strife---that there are at least three different sorts of tensor
differential category which have been discussed in the literature:
\begin{enumerate}[(A)]
\item A differential given by a {\bf deriving transformation} as introduced in \cite{diffl}.
\item A differential given by a  {\bf codereliction} as introduced \cite{ER} and described in \cite{diffl}
\item A differential given by a {\bf creation operator} as introduced in \cite{fiore}.
\end{enumerate}
Each form of differential is more specialized than its predecessor: the first merely requires a coalgebra modality, the second requires a
bialgebra modality, while the last requires a storage modality.   A standard way in which a bialgebra modality arises is from the presence of
biproducts together with a storage modality: the Seely isomorphism then transfers the bialgebra structure of the biproduct onto the modality.
This means that the last two forms of differential are particularly suited to storage settings.  As, in the current context, we are considering
storage categories one might think that the differential which one extracts from a Cartesian differential storage category should  be of type (C) or at least (B).  It is, therefore,
perhaps somewhat unexpected that the differential that emerges is one of type (A).

Cartesian differential storage categories are important in their own right.  Their structure can be approached in two very different ways: from the perspective of being a Cartesian differential category with tensorial representation,  or as the coKleisli category of a tensor differential category.  This tension allows one to transfer arguments between the Cartesian and the tensor worlds.  This ability is frequently used in differential geometry where one often wants to view differential forms as maps from tensor products.


\subsection{Cartesian differential categories}

Cartesian differential categories, were introduced in \cite{CartDiff}: here we briefly review their properties.

One way to view the axiomatization of Cartesian differential categories is as an abstraction of the Jacobian of a smooth
map $f: \mathbb{R}^n \to \mathbb{R}^m$.  One ordinarily thinks of the Jacobian as a smooth map
	\[ J(f): \mathbb{R}^n \to \mbox{Lin}(\mathbb{R}^n, \mathbb{R}^m).  \]
Uncurrying, this means the Jacobian can also be seen as a smooth map
	\[ J(f): \mathbb{R}^n \times \mathbb{R}^n \to \mathbb{R}^m \]
which is linear in its first variable.  A Cartesian differential category asks for an operation of this type, satisfying the axioms described below.  However, notice that to express the axioms, one needs the ability to add parallel maps.  It turns out this is not an enrichment in commutative monoids as one might expect, but rather a skew enrichment, \cite{street}, which makes it a Cartesian left additive category.  In such a category the addition of arrows is only preserved by composition on the left, that is  $f (g+h) = fg + gh$ and $f 0 = 0$ .

\begin{definition}
A \textbf{Cartesian differential category} is a Cartesian left additive category with an operation
$$\infer[{D[\_]}]{X \x X \to_{D[f]} Y}{X \to^f Y}$$
(called ``differentiation'') satisfying:
\begin{enumerate}[{\bf [CD.1]}]
\item $D[f+g] = D[f]+D[g]$ and $D[0]=0$ (differentiation preserves addition);
\item $\< a+b,c\>D[f] = \< a,c\>D[f] + \< b,c\>D[f]$ and $\< 0,a\>D[f] = 0$ (a derivative is
additive in its first variable);
\item $D[\pi_0] = \pi_0\pi_0$, and $D[\pi_1]= \pi_0\pi_1$ (projections and identity are linear);
\item $D[\<f,g\>] = \< D[f],D[g]\>$ (differentiation is compatible with pairing) ;
\item $D[fg] = \<D[f],\pi_1f \>D[g]$  (the chain rule);
\item $\<\<a,0\>,\<c,d\>\>D[D[f]] = \<a,d\>D[f]$ (the differential is linear);
\item $\< \< a,b\>,\< c,d\>\> D[D[f]] = \<\< a,c\>,\< b,d\>\> D[D[f]]$ (interchange rule);
\end{enumerate}
\end{definition}

We recall a number of examples of Cartesian differential categories:

\begin{example}{\em ~
\begin{enumerate}[(1)]
\item If $\X$ is an additive Cartesian category ({\em i.e.} enriched in commutative monoids) then, by defining $D[f] = \pi_0f$, it can be viewed as a Cartesian differential category in which
every map is linear.
\item Smooth functions on finite dimensional Euclidean vector spaces form a Cartesian differential category.
\item The coKleisli category of any tensor differential category (satisfying the interchange law) is a Cartesian differential category, \cite{CartDiff}.  A basic example of a tensor differential category is
provided by the category of relations, ${\sf Rel}$, with respect to the ``multi-set'' (or ``bag'') comonad.   The deriving transformation $d_\ox: A \ox M(A) \to M(A)$, which
adds another element to the multi--set $M(A)$,  provides the (tensor) differential structure.
\item Convenient vector spaces \cite{BET} form a Cartesian differential category.
\item There is a comonad ${\sf Faa}$ on the category of Cartesian left additive categories whose coalgebras are exactly Cartesian differential categories \cite{faa}.
\end{enumerate} }
\end{example}

Cartesian differential categories already have a notion of ``linear map'': namely those $f$ such that $D[f] = \pi_0f$.
Furthermore, it is a basic result that any simple slice, $\X[A]$, of a Cartesian
differential category, $\X$, is also a Cartesian differential category.  The differential in the simple slice $\X[A]$ is given by:
$$\infer[{D_A[\_]}]{(X \x X) \x A \to_{D_A[f]} Y}{X \x A \to^f Y}$$
where
\begin{eqnarray*}
\lefteqn{(X \x X) \x A \to^{D_A[f]} Y} \\
& = &  (X \x X) \x A \to_{\< \< \pi_0\pi_0,0\>,\<\pi_0\pi_1,\pi_1\>\>}  (X \x A) \x (X \x A) \to_{D[f]} Y
\end{eqnarray*}
This is precisely the familiar notion of a {\em partial derivative\/}.

In particular, this means we may isolate linear maps in this differential sense in each slice as those maps with $D_A[f] = (\pi_0 \x 1) f$.   Our first observation is then:

\begin{proposition}
In a Cartesian differential category, $\X$, the linear maps form a system of linear maps.
\end{proposition}

\begin{proof}
We must check the requirements of being a linear system:
\begin{enumerate}[{\bf [LS.1]}]
\item We require, for each $A \in \X$, that the identity maps,
projections and pairings of linear maps are linear.  As explained above, in
$\X[A]$ a map $f: X \x A \to Y$ is linear in case $D_A[f] = (\pi_0 \x 1)f :
(X \x X) \x A \to Y$  where $D_A[f]$ is the
partial derivative of $f$:
$$D_A[f] :=  \< \< \pi_0\pi_0,0\>,\<\pi_0\pi_1,\pi_1\> \> D[f]$$
We shall leave all but the last requirement to the reader for which we
shall do a concrete calculation in $\X$.   Suppose $f$ and $g$ are linear
in this sense in $\X[A]$ (that is they are in ${\cal L}[A]$), then we
must show that the pairing of $f$ and $g$ in $\X[A]$ is linear: as a map in
$\X$ this pairing is just $\<f,g\>$ so that:
\begin{eqnarray*}
D_A[\< f,g\>] & = & \< \< \pi_0\pi_0,0\>,\<\pi_0\pi_1,\pi_1\> \> D[\<f,g\>] \\
                    & = & \< \< \pi_0\pi_0,0\>,\<\pi_0\pi_1,\pi_1\> \> \< D[f],D[g] \> \\
                    & = &  \< \< \< \pi_0\pi_0,0\>,\<\pi_0\pi_1,\pi_1\>
\>D[f],\< \< \pi_0\pi_0,0\>,\<\pi_0\pi_1,\pi_1\> \>D[g] \> \\
                    & = &  \< D_A[f],D_A[g]\> \\
                    & & ~~~\mbox{(Using linearity of $f$ and $g$)} \\
                    & = & \< \< (\pi_0 \x 1) f, (\pi_0 \x 1) g \> \\
                    & = & (\pi_0 \x 1)\<f,g\>
\end{eqnarray*}
showing that $\<f,g\>$ is linear in this sense.
\item We must show that these linear maps are closed to composition in each
slice and that when $g$ is a linear retraction that $gh \in {\cal L}[A]$
implies $h \in {\cal L}[A]$.   Leaving composition to the reader let us
focus on the second part.  If $sg = 1$ then using the non-trivial fact that
$D_A[\_]$ is a differential in $\X[A]$ we have
\begin{eqnarray*}
D_A[h] & = & D_A[sgh] = \< D_A[s],\pi_1s\>D_A[gh] = \< D_A[s],\pi_1s\>
\pi_0gh \\
            & = & D_A[s]g h = D_A[sg]h = D_A[1]h  = \pi_0h.\\
\end{eqnarray*}
\item  For the last part we shall do a concrete calculation in $\X$:   we
must show that if $h$ is linear in $\X[B]$ then
$\X[f](h)$ is linear in $\X[A]$; here is the calculation:
\begin{eqnarray*}
D_A[\X[f](h)] & = & D_A[(1\x f)h] = \<\<
\pi_0\pi_0,0\>,\<\pi_0\pi_1,\pi_1\> \> D[(1 \x f)h] \\
                     & = & \< \< \pi_0\pi_0,0\>,\<\pi_0\pi_1,\pi_1\> \> \<
D[1 \x f], \pi_1 (1 \x f) \> D[h] \\
                     & = & \< \< \pi_0\pi_0,0\>,\<\pi_0\pi_1,\pi_1\> \> \<
\< \pi_0\pi_0,\<\pi_0\pi_1,\pi_1\pi_1\> D[f]\>, \pi_1 (1 \x f) \> D[h] \\
                     & = & \< \< \pi_0\pi_0,0\>,\<\pi_0\pi_1,\pi_1\> \> \<
\< \pi_0\pi_0,\<\pi_0\pi_1,\pi_1\pi_1\> D[f]\>, \pi_1 (1 \x f) \> D[h] \\
                     & = & \< \< \pi_0\pi_0,\< 0,\pi_1\>D[f] \>
,\<\pi_0\pi_1,\pi_1\>(1 \x f)\>D[h] \\
                     & = &  \< \< \pi_0\pi_0,0 \> ,\<\pi_0\pi_1,\pi_1\>(1
\x f)\>D[h] \\
                     & = & (1 \x f) \<\< \pi_0\pi_0,0 \>
,\<\pi_0\pi_1,\pi_1\>\> D[h] \\
                     & & ~~~\mbox{(Linearity of $h$)} \\
                     & = & (1 \x f) (\pi_0 \x 1) h \\
                     & = & (\pi_0 \x 1) (1 \x f) h
\end{eqnarray*}
\end{enumerate}
\end{proof}


\subsection{Cartesian differential storage categories: the basics}

A {\bf Cartesian differential storage category}
is a Cartesian differential category whose linear maps---in the natural
differential sense---are (strongly and persistently) classified.

We first observe that every Cartesian  differential category has a codereliction map defined by:
$$\eta_A := \< 1,0\> D[\varphi]: A \to S(A).$$

\begin{lemma}
$\eta$ so defined is a codereliction map: that is it is natural for linear maps and has $\eta\epsilon = 1$.
\end{lemma}

\begin{proof}~
We need to show $\eta_A \epsilon_A = 1_A$:
\begin{eqnarray*}
\eta_A \epsilon_A & = & \< 1,0\> D[\varphi] \epsilon \\
& = & \< 1,0\> D[\varphi]\epsilon] ~~~~~\mbox{(as  $\epsilon$ is linear)} \\
& = & \< 1,0\> D[1] = \<1,0\>\pi_0 = 1
\end{eqnarray*}
and that, if $f: A \to B$ is a linear map, that $f \eta_B = \eta_A S(f)$:
\begin{eqnarray*}
\eta_A S(f) & = & \< 1,0\> D[\varphi] S(f) \\
& = & \< 1,0\> D[\varphi S(f)] ~~~~~\mbox{(as  $S(f)$ is linear)} \\
& = & \< 1,0\> D[f \varphi] \\
& = & \<1,0\>  \< D[f],\pi_1f\> D[\varphi] \\
& = & \<1,0\>  \< \pi_0f,\pi_1f\> D[\varphi] \\
& = & \<1,0\>  \< \pi_0f,\pi_1f\> D[\varphi] \\
& = & \< f,0f \> D[\varphi] = f \<1,0\>D[\varphi] = f \eta_B
\end{eqnarray*}
\end{proof}

Because Cartesian differential storage categories always have a codereliction map, up to
splitting linear idempotents, they also have tensor representation.
Thus it makes sense to ask whether the subcategory of linear maps forms
a tensor differential category.  To prove this is so will be the main aim of
the current section.  However, before turning to this it is worth making some
further basic observations.

Suppose $\X$ is a Cartesian storage category which is also a Cartesian
differential category.  In this context, $\X$, has two notions of
``linear'' map: {\em viz} the notion which all storage categories have,
and the notion that all Cartesian differential categories have.
An interesting observation is:

\begin{proposition}
$\X$ is a Cartesian differential storage category if and only if it is both a Cartesian differential and storage category and $S(f)$ (for any $f$)
and $\epsilon$ are linear in the differential sense and the codereliction map $\eta$ is linear in the storage sense.
\end{proposition}

\begin{proof}
By assumption, $D[S(f)]=\pi_0 S(f)$, for any $f$,
$D[\epsilon]=\pi_0 \epsilon$, and $\eta$ is $\epsilon$-natural.  We
must show that $f$ is $\epsilon$-natural if and only if $D_\x[f]=\pi_0 f$.
So suppose first that $f$ is $\epsilon$-natural:
\begin{eqnarray*}
D[f] &=& (\varphi\x\varphi)(\epsilon\x\epsilon)D[f] \\
&=& (\varphi\x\varphi) D[\epsilon f] \\
&=& (\varphi\x\varphi) D[S(f)\epsilon] \\
&=& (\varphi\x\varphi) \pi_0 S(f)\epsilon \\
&=& (\varphi\x\varphi) \pi_0 \epsilon f \\
&=& \pi_0 \varphi \epsilon f = \pi_0 f
\end{eqnarray*}
Next, suppose that $D[f]=\pi_0 f$ we must show $f$ is $\epsilon$-natural assuming that $\eta$ is.  We observe
that, in this case, $f = \eta S(f) \epsilon$, from which the result is immediate, as $\eta$, $S(f)$, and $\epsilon$ are $\epsilon$ natural.
Here is the calculation:
\begin{eqnarray*}
f &=& \<1,0>\pi_0 f \\
&=& \<1,0>D[f] \\
&=& \<1,0>D[ f \varphi \epsilon] \\
&=& \<1,0>D[ \varphi S(f) \epsilon] \\
&=& \<1,0>D[ \varphi] S(f) \epsilon \\
&=& \eta S(f) \epsilon .
\end{eqnarray*}
\end{proof}

Recall that in a tensor differential category it is possible to re-express the derivative using a {\em deriving transformation}
$d_\ox: A \ox S(A) \to S(A)$ into a more compact form.  It is natural to wonder whether the same thing cannot be done for
Cartesian differential storage categories: that is define the derivative in terms of analogous structure:

\begin{definition}
A {\bf Cartesian deriving transformation} on a Cartesian storage
category is a (not-necessarily natural) transformation
$d_{\x}:A \x A \to S(A)$ satisfying:
\begin{enumerate}[{\bf[cd.1]}]
\item $d_\x S(0)\epsilon=0$, $d_\x S(f+g)\epsilon =
d_\x(S(f)+S(g))\epsilon$
\item $<h+k,v>d_\x=<h,v>d_\x+<k,v>d_\x$, $<0,v>d_\x=0$
\item $d_\x\epsilon=\pi_0$
\item $d_\x S(<f,g>)\epsilon=d_\x<S(f)\epsilon,S(g)\epsilon>$ (Note that
$d_\x S(!)\epsilon=d_\x!=!$ is true since $1$ is terminal.)
\item $d_\x S(fg)\epsilon=<d_\x S(f)\epsilon,\pi_1f>d_\x S(g)\epsilon$
\item $<<g,0>,<h,k>>d_\x S(d_\x)\epsilon=<g,k>d_\x$
\item
$<<0,h>,<g,k>>d_\x S(d_\x)\epsilon=<<0,g>,<h,k>>d_\x S(d_\x)\epsilon$
\item $\eta=_{\sf def} <1,0>d_\x$ is $\epsilon$-natural (or linear).
\end{enumerate}
\end{definition}

It is now straightforward to observe:

\begin{proposition}
A Cartesian storage category with a Cartesian deriving transformation is precisely a Cartesian differential storage category.
\end{proposition}

\begin{proof}
The translation between the two structures is given by:
\[ D[f] := d_\x S(f)\epsilon \mbox{ ~~~and~~~ } d_\x := D[\varphi] \]
Note that these are inverse as
$$D[f] := d_\x S(f)\epsilon := D[\varphi]S(f)\epsilon = D[\varphi S(f)\epsilon] = D[f \varphi\epsilon] = D[f]$$
and
$$d_\x := D[\varphi] := d_\x S(\varphi) \epsilon = d_\x.$$

Most of the axioms are clearly direct translations of each other: {\bf[CD.1,2,4-7]} and
{\bf[cd.1,2,4-7]} are clearly equivalent through this translation. For {\bf[CD.3]}, note that
$D[1]=d_\x \epsilon=\pi_0$ and that $D[\pi_i]=d_\x S(p_i) \epsilon = d_\x \epsilon \pi_i = \pi_0\pi_i$, since in a
storage category, projections are linear (by {\bf[LS.1]}).

It remains to show that being linear in the differential sense coincides with being linear in the storage sense.
First note if $f$ is epsilon natural, that is $\epsilon f = S(f)\epsilon$, then  $D[f]=d_\x S(f) \epsilon = d_\x \epsilon f = \pi_0 f$.
Conversely, suppose that $D_\x[f]=\pi_0 f$, then $f = \eta S(f) \epsilon$ as
$$ f ~=~ \<1,0>\pi_0f ~=~ \<1,0\>d_\x S(f)\epsilon ~=~ \eta S(f)\epsilon$$
which by assumption is linear (in the sense of being $\epsilon$-natural).
\end{proof}


\subsection{The main theorem}

From the results of Section 3, we also note that, in the presence of a codereliction map and when  sufficient linear idempotents split, every Cartesian
differential storage category has strong persistent tensor representation.    The main result of the paper is:

\begin{theorem}  \label{lin-of-cdsc}
The linear maps of a Cartesian differential storage category, in which linear idempotents split, form a tensor storage differential category
satisfying the interchange rule.
\end{theorem}

Let us first recall what a tensor differential category is.   A {\bf
tensor differential category} is a tensor category with a coalgebra
modality (see Definition \ref{coalg-modality}) equipped with a natural
transformation
$$d_\ox: A \ox S(A) \to S(A)$$
called a {\bf deriving transformation} satisfying:
\begin{enumerate}[{\bf [d.1]}]
\item $d_\ox e = 0$ (constants)
\item $d_\ox \epsilon = 1 \ox e$ (linear maps)
\item $d_\ox \Delta = (1 \ox \Delta) (d_\ox \ox 1) + (1 \ox \Delta)(c_\ox \ox 1)(1 \ox d_\ox)$ (the product rule)
\item $ d_\ox \delta = (1 \ox \Delta_\ox) a_\ox (d_\ox \ox \delta) d_\ox$ (the chain rule)
\item $(1 \ox d_\ox) d_\ox = a_\ox(c_\ox \ox 1)a_\ox^{-1} (1 \ox d_\ox) d_\ox$ (the  interchange rule)
\end{enumerate}

A tensor {\em storage} differential category simply means that the
storage transformation (see Definition \ref{storage-trans}) are
isomorphisms.  Note also that in \cite{diffl} a differential category
was defined to be one satisfying only the first four of these
conditions: the last condition, the interchange rule, was introduced in
\cite{CartDiff} in order to ensure that the coKleisli category was a
Cartesian differential category.  Here we shall simply add the
interchange law as a condition for being a tensor differential category
as it will turn out not only to be sufficient but necessary to obtain
the characterization of tensor storage differential categories  (with an
exact modality) as the linear maps of a Cartesian differential storage
category.

The remainder of the section is dedicated to proving Theorem \ref{lin-of-cdsc}.

Given a Cartesian differential storage category here is the definition of the tensor deriving transformation:
$$\xymatrix{A \x S(A) \ar[d]_{\varphi_\ox}  \ar[r]^{\eta \x 1} & S(A) \x S(A) \ar[r]^{m_\x} & S(A \x A) \ar[r]^{S(D[\varphi_A])} & S^2(A) \ar[d]^{\epsilon} \\
A \ox S(A) \ar@{..>}[rrr]_{d_\ox} &&&  S(A)}$$

We start by observing:

\begin{lemma}~
\begin{enumerate}[(i)]
\item $d_\ox$ is natural for linear maps;
\item
The tensor deriving transformation and the differential are inter-definable with
$$D[f] = (1 \x \varphi)\varphi_\ox d_\ox S(f)\epsilon.$$
\end{enumerate}
\end{lemma}

\begin{proof}~
\begin{enumerate}[(i)]
\item  Note that  $d_\ox S(f) = (f \ox S(f)) d_\ox$ if and only if $\varphi_\ox d_\ox S(f) = \varphi_\ox (f \ox S(f)) d_\ox$
but $\varphi d_\ox = (\eta \x 1) m_\x S(D[\varphi])\epsilon$ which is natural for linear $f$. This provides the result immediately.
\item It suffices to prove the equality:
\begin{eqnarray*}
\lefteqn{(1 \x \varphi)\varphi_\ox d_\ox S(f)\epsilon} \\
& = & (1 \x \varphi) (\eta \x 1) m_\x S(D[\varphi]) \epsilon S(f)\epsilon \\
& = & (1 \x \varphi) (\eta \x 1) m_\x S(D[\varphi ]) S(S(f) \epsilon) \epsilon] \\
& = & (\eta \x 1) (1 \x \varphi) m_\x S(D[\varphi S(f) \epsilon]) \epsilon \\
& = & (\eta \x 1) \theta S(D[f \varphi \epsilon])  \epsilon ~~~~ \mbox{(strength)}\\
& = & (1 \x \varphi) ( 1 \x \epsilon) D[f] ~~~~ \mbox{(linearity of differential in first argument)}\\
& = & D[f] \\
\end{eqnarray*}
\end{enumerate}
\end{proof}

We want to show that $d_\ox$ as defined satisfies {\bf (d.1)}--{\bf (d.5)} above.  So we shall simply go through the conditions:

\begin{enumerate}[{\bf [d.1]}]
\item Constants ($d_\ox e = 0$)
\begin{eqnarray*}
d_\ox e & = & d_\ox S(0)s_\top \\
              & = & (0 \ox S(0))d_\ox s_\top \\
              & = & 0
\end{eqnarray*}

\item Differentials of linear maps ($d_\ox \epsilon = 1 \ox e$)
\begin{eqnarray*}
\varphi_\ox d_\ox \epsilon & = & (\eta \x 1) m_\x S(D[\varphi]) \epsilon\epsilon \\
      & = & (\eta \x 1) m_\x S(D[\varphi] \epsilon) \epsilon \\
      & = & (\eta \x 1) m_\x S(D[\varphi \epsilon]) \epsilon \\
      & = & (\eta \x 1) m_\x S(D[1]) \epsilon \\
      & = & (\eta \x 1) m_\x S(\pi_0) \epsilon \\
      & = & (\eta \x 1) \pi_0 \epsilon \\
      & = & \pi_0 \eta \epsilon \\
      & = & \pi_0 = \varphi_\ox (1 \ox e)
\end{eqnarray*}

\item The product rule  ($d_\ox \Delta = (1 \ox \Delta) (d_\ox \ox 1) + (1 \ox \Delta)(c_\ox \ox 1)(1 \ox d_\ox)$)

This is more complicated.  Start by noting:
\begin{eqnarray*}
d_\ox \Delta & = & d_\ox S(\<1,0\> + \< 0,1\>) s_\ox \\
                       & = & ((\<1,0\> + \< 0,1\>) \ox S(\<1,0\> + \< 0,1\>)) d_\ox s_\ox \\
                       & = & (\< 1,0\>\ox S(\<1,0\> + \< 0,1\>)) d_\ox s_\ox \\
                       & & ~~~~~ + (\< 0,1\> \ox S(\<1,0\> + \< 0,1\>)) d_\ox s_\ox
\end{eqnarray*}

It suffices to prove that
$$(\< 1,0\>\ox S(\<1,0\> + \< 0,1\>)) d_\ox s_\ox = (1 \ox \Delta) (d_\ox \ox 1)~~~~\mbox{ and }$$
$$(\< 0,1\> \ox S(\<1,0\> + \< 0,1\>)) d_\ox s_\ox = (1 \ox \Delta)(c_\ox \ox 1)(1 \ox d_\ox)$$
We shall demonstrate the former leaving the latter to the reader.
However, we require some observations first:

\begin{lemma}~ \label{more-bilinear-identities}
\begin{enumerate}[(i)]
\item If $h: A \x B \to C$ is bilinear, then $\<\<a,b\>,\<c,e\>\> D[h] = \<a,e \> h+ \<c,b\>h$;
\item $\<\<a,b\>,\<c,e\>\> D[m_\x] = \<a,e \> m_\x + \<c,b\>m_\x$;
\item $\<\<a,b\>,\<c,e\>\> D[\varphi_{A \x B}] = \< \< a,c\>D[\varphi_A], e\varphi_B\>m_\x + \< c \varphi_A,\<b,e\>D[\varphi_B]\>m_\x$;
\item $(\<1,0\> \x 1)D[\varphi_{A \x B}] = a_\x (D[\varphi_A] \x \varphi_B)m_\x$.
\end{enumerate}
\end{lemma}

\begin{proof}~
\begin{enumerate}[(i)]
\item  We have the following calculation:
\begin{eqnarray*}
\<\<a,b\>,\<c,e\>\> D[h] & = & \<\<a,0\>,\<c,e\>\> D[h] + \<\<0,b\>,\<c,e\>\> D[h] \\
  & = & \< a,\<c,e\>\> D_B[h] + \< b,\<c,e\>\> D_A[h] \\
  & = & \< a,e \> h + \< c,b\>h ~~~~~\mbox{($h$ is bilinear)}.
\end{eqnarray*}
\item  Use the fact that $m_\x$ is bilinear (see Lemma \ref{lem:basic_tensor_rep}).
\item  Using the fact that  $\varphi_{A \x B} = (\varphi_A \x \varphi_B) m_\x$ (see Proposition \ref{prop:linmaps}) we have:
\begin{eqnarray*}
\lefteqn{<\<a,b\>,\<c,e\>\> D[\varphi_{A \x B}]} \\ & = &
      <\<a,b\>,\<c,e\>\> D[(\varphi_A \x \varphi_B)m_\x] \\
      & = & <\<a,b\>,\<c,e\>\> \< D[\varphi_A \x \varphi_B], \pi_1 (\varphi_A \x \varphi_B) \> D[m_\x] \\
      & = & <\<a,b\>,\<c,e\>\> \< \<  (\pi_0 \x \pi_0)D[\varphi_A], (\pi_1 \x \pi_1)D[\varphi_B] \>, \< \pi_1 \pi_0 \varphi_A ,\pi_1 \pi_1 \varphi_B \> \> D[m_\x] \\
      & = & \< \<  \< a,c\> D[\varphi_A], \< b,e\> D[\varphi_B] \>, \< c \varphi_A ,e \varphi_B \> \> D[m_\x] \\
      & = & \<  \< a,c\> D[\varphi_A],e \varphi_B \>m_\x + \< c \varphi_A, \< b,e\> D[\varphi_B]\> m_\x.
 \end{eqnarray*}
\item We apply (iii):
\begin{eqnarray*}
\lefteqn{ (\<1,0\> \x 1)D[\varphi_{A \x B}]} \\
& =& \<\<\pi_0,0\>,\<\pi_1\pi_0,\pi_1\pi_1\>\> D[\varphi_{A \x B}] \\
& = & \<  \< \pi_0,\pi_1\pi_0\> D[\varphi_A],\pi_1\pi_1 \varphi_B \>m_\x + \< \pi_1\pi_1 \varphi_A, \< 0,\pi_1\pi_1\> D[\varphi_B]\> m_\x \\
& = & \<  \< \pi_0,\pi_1\pi_0\> D[\varphi_A],\pi_1\pi_1 \varphi_B \>m_\x + \< \pi_1\pi_1 \varphi_A, 0> m_\x \\
& =  & \<  \< \pi_0,\pi_1\pi_0\> D[\varphi_A],\pi_1\pi_1 \varphi_B \>m_\x + 0 \\
& = & a_\x (D[\varphi_A] \x \varphi_B) m_\x.
\end{eqnarray*}
\end{enumerate}
\end{proof}

We are now ready to calculate.  Since both maps are linear we may prefix each side with the universal map which gives tensorial representation;
the universal property tells us that the maps are equal if and only if these composites are equal.  Here is the calculation:
\begin{eqnarray*}
\lefteqn{ \varphi_\ox (\< 1,0\>\ox S(\<1,0\> + \< 0,1\>)) d_\ox s_2}\\
 & = & (\< 1,0\>\x S(\Delta_\x)) \varphi_\ox d_\ox s_2 \\
& = & (\< 1,0\>\x S(\Delta_\x)) (\eta \x 1)m_\x S(D[\varphi_{A \x A}]) \epsilon s_2 \\
& = & (\eta \x S(\Delta_\x))m_\x S((\< 1,0\> \x 1)D[\varphi_{A \x A}]) \epsilon s_2 \\
& = & (\eta \x \Delta_\x)(1 \x m_\x) m_\x S(a_\x (D[\varphi_A] \x \varphi_A)m_\x) \epsilon s_2 ~~~\mbox{(Lemma \ref{more-bilinear-identities} (iv))}\\
& = & (\eta \x \Delta_\x) a_\x (m_\x \x 1) m_\x S(D[\varphi_A] \x \varphi_A) S(m_\x) \epsilon s_2 \\
& = & (\eta \x \Delta_\x)a_\x (m_\x \x 1) (S(D[\varphi_A]) \x S(\varphi_A)) m_\x S(m_\x) \epsilon s_2 \\
& = & (\eta \x \Delta_\x)a_\x(m_\x \x 1) (S(D[\varphi_A]) \x S(\varphi_A)) (\epsilon \x \epsilon) m_\x s_2  ~~~\mbox{(Lemma \ref{bilinear-identities} (ii))}\\
& = & (\eta \x \Delta_\x)a_\x( m_\x \x 1) (S(D[\varphi_A]) \x 1) (\epsilon \x 1) \varphi_\ox \\
& = & (1 \x \Delta_\x)a_\x (((\eta \x 1) m_\x S(D[\varphi_A])\epsilon) \x 1) \varphi_\ox \\
& = & (1 \x \Delta_\x)a_\x ((\varphi_\ox d_\ox) \x 1) \varphi_\ox \\
& = & (1 \x \Delta_\x)a_\x (\varphi_\ox \x 1)\varphi_\ox (d_\ox \ox 1)  \\
& = & (1 \x \Delta_\ox) \varphi_\ox (d_\ox \ox 1)  \\
& = & \varphi_\ox (1 \ox \Delta_\ox) (d_\ox \ox 1)
\end{eqnarray*}

\item To prove the chain rule, $ d_\ox \delta = (1 \ox \Delta_\ox) a_\ox (d_\ox \ox \delta) d_\ox$, we
precompose both sides with $(1 \x \varphi) \varphi_\ox$.  As both maps are clearly linear we can then use the
universal properties of $\varphi$ and $\varphi_\ox$ to conclude the original maps are equal if and only if these composites are:

\begin{eqnarray*}
\lefteqn{(1 \x \varphi) \varphi_\ox d_\ox \delta} \\
 & = & (1 \x \varphi) (\eta \x 1)m_\x S(D[\varphi])\epsilon S(\varphi) \\
& = & (1 \x \varphi)  (\eta \x 1)m_\x S(D[\varphi]S(\varphi))\epsilon  \\
& = & (\eta \x \varphi) m_\x S(D[\varphi S(\varphi)])\epsilon  \\
& = & (\eta \x 1) \theta S(D[\varphi\varphi])\epsilon  \\
& = & (\eta \x 1) (\epsilon \x 1) D[\varphi\varphi]  \\
& = & D[\varphi\varphi] \\
\lefteqn{(1 \x \varphi) \varphi_\ox (1 \ox \Delta_\ox) a_\ox (d_\ox \ox \delta) d_\ox} \\
& = & (1 \x \varphi) (1 \x \Delta_\ox) \varphi_\ox  a_\ox (d_\ox \ox \delta) d_\ox \\
& = & (1 \x \varphi\Delta_\x) a_\x(\varphi_\ox \x 1) \varphi_\ox  (d_\ox \ox \delta) d_\ox  \\
& = & (1 \x \varphi\Delta_\x) a_\x((\varphi_\ox d_\ox) \x S(\varphi)) \varphi_\ox d_\ox \\
& = & (1 \x \varphi\Delta_\x) a_\x(((\eta \x 1)m_\x S(D[\varphi])\epsilon) \x S(\varphi)) (\eta \x 1)m_\x S(D[\varphi])\epsilon \\
& = & (1 \x \Delta_\x (\varphi \!\x\! \varphi)) a_\x ((\eta \!\x\! 1)m_\x S(D[\varphi]) \epsilon \!\x\! S(\varphi)) (\eta \!\x\! 1)m_\x S(D[\varphi] ) \epsilon\\
& = & (1 \x \Delta_\x) a_\x (((\eta \x \varphi)m_\x S(D[\varphi]) \epsilon) \x \varphi S(\varphi)) (\eta \x 1)m_\x S(D[\varphi] ) \epsilon\\
& = & (1 \x \Delta_\x) a_\x (((\eta \x 1)\theta S(D[\varphi]) \epsilon) \x (\varphi\varphi)) (\eta \x 1)m_\x S(D[\varphi] ) \epsilon\\
& = & (1 \x \Delta_\x) a_\x (((\eta \x 1)(\epsilon \x 1)D[\varphi]) \x (\varphi\varphi)) (\eta \x 1)m_\x S(D[\varphi] ) \epsilon\\
& = & (1 \x \Delta_\x) a_\x (D[\varphi] \x \varphi) (\eta \x \varphi)m_\x S(D[\varphi] ) \epsilon\\
& = & (1 \x \Delta_\x) a_\x (D[\varphi] \x \varphi) (\eta \x 1)(\epsilon \x 1)D[\varphi] \\
& = &(1 \x \Delta_\x) a_\x (D[\varphi] \x \varphi) D[\varphi] \\
& = & D[\varphi\varphi]
\end{eqnarray*}

\item  Interchange  ($(1 \ox d_\ox) d_\ox = a_\ox(c_\ox \ox 1)a_\ox^{-1} (1 \ox d_\ox) d_\ox$)

We may precompose with $(1 \x 1 \x \varphi) (1 \x \varphi_\ox)\varphi_\ox$.

\begin{eqnarray*}
\lefteqn{\<a,b,c\>(1 \x 1 \x \varphi) (1 \x \varphi_\ox)\varphi_\ox(1 \ox d_\ox) d_\ox}\\
& = & \<a,b,c\>(1 \x 1 \x \varphi) (1 \x \varphi_\ox d_\ox)\varphi_\ox d_\ox \\
& = & \<a,b,c\>(1 \x 1 \x \varphi) (1 \x (\eta \!\x\! 1)m_\x S(D[\varphi]) \epsilon) (\eta \!\x\! 1)m_\x S(D[\varphi]) \epsilon \\
& = & \<a \eta,b \eta,c\> 1 \x \theta (S(D[\varphi]) \epsilon) m_\x S(D[\varphi]) \epsilon \\
& = & \<a \eta,b \eta,c\> (1 \x (\epsilon \x 1)D[\varphi] ) m_\x S(D[\varphi]) \epsilon \\
& = & \< a \eta,\<b,c\>D[\varphi]\>m_\x S(D[\varphi]) \epsilon
\end{eqnarray*}

The identity requires us to show that  $a$ and $b$  can be swapped.

Consider the double differential of $\varphi$:
\begin{eqnarray*}
\<\< 0,a\>,\<b,c\>\>D[D[\varphi]] & = &\< \< 0,a\>,\<b,c\>\>D[\varphi S(D[\varphi])\epsilon] \\
                     & = & \< 0,a\>,\<b,c\>\>D[\varphi] S(D[\varphi])\epsilon \\
                     & = & \< 0,a\>,\<b,c\>\>D[(\varphi \x \varphi)m_\x]S(D[\varphi])\epsilon \\
                     & = & \< \< 0,a\>,\<b,c\>\>D[\varphi \x \varphi], \< b \varphi, c \varphi \> \> D[m_\x] S(D[\varphi])\epsilon  \\
                     & = & \< \< 0,b\> D[\varphi], \< a,c\>D[\varphi] \> ,\< b \varphi, c \varphi \> \> D[m_\x] S(D[\varphi])\epsilon  \\
                     & = & ( \<  \< 0,b\> D[\varphi],c \varphi \> m_\x + \< b \varphi, \< a,c\>D[\varphi] \> m_\x ) S(D[\varphi])\epsilon  \\
                     & = & ( 0 + \< b \varphi, \< a,c\>D[\varphi] \> m_\x ) S(D[\varphi])\epsilon  \\
                     & = & \< b \varphi, \< a,c\>D[\varphi] \> m_\x  S(D[\varphi])\epsilon
\end{eqnarray*}
Then we have:
\begin{eqnarray*}
\lefteqn{\< a \varphi,\<b,c\>D[\varphi] \> m_\x S(D[\varphi]) \epsilon }\\
& = & \<0,a\>,\<b,c\>\>D[D[\varphi]] \\
& = & \<0,b\>,\<a,c\>\>D[D[\varphi]] \\
& = & \< b \varphi,\<a,c\>D[\varphi] \> m_\x S(D[\varphi]) \epsilon
\end{eqnarray*}
We are nearly there except we would like to change $\varphi$ into $\eta$, which is allowed by another lemma.

\begin{lemma}~
$$\< b \varphi,\<a,c\> D[\varphi] \>m_\x S(D[\varphi])\epsilon = \< a \eta,\<b,c\>D[\varphi]\>m_\x S(D[\varphi])\epsilon$$
\end{lemma}

To prove this lemma we partially differentiate both sides of the equation derived above in the form:
$$\< \pi_0\pi_0 \varphi,\<\pi_0\pi_1,\pi_1\>D[\varphi] \> m_\x S(D[\varphi]) \epsilon = \< \pi_0\pi_1 \varphi,\<\pi_0\pi_0,\pi_1\>D[\varphi] \> m_\x S(D[\varphi]) \epsilon$$
with respect to the first coordinate at position $0$.  This is best done using the term logic.  In the derivation below we flip between the term logic and the categorical term.
The derivation may best be read from the middle outward:

\begin{eqnarray*}
\lefteqn{\< \pi_0\pi_0 \eta,\<\pi_0\pi_1,\pi_1\>D[\varphi]\>m_\x S(D[\varphi])\epsilon} \\
& = & \left\llbracket ((a,b),c) \mapsto \epsilon(S(D[\varphi])(m_\x(\eta(a),D[\varphi](c) \cdot b))) \right\rrbracket \\
& = & \left\llbracket ((a,b),c) \mapsto \epsilon(S(D[\varphi])(m_\x(\diff{\varphi(x)}{x}{0}{a},D[\varphi](c) \cdot b))) \right\rrbracket \\
& = & \left\llbracket ((a,b),c) \mapsto \diff{\epsilon(S(D[\varphi])(m_\x(\varphi(x),D[\varphi](c).b)))}{x}{0}{a} \right\rrbracket \\
& = & \<\< \pi_0\pi_0,0,\>,0\>,\<\<0,\pi_0\pi_1\>,\pi_1\>\> D[\< \pi_0\pi_0 \varphi,\<\pi_0\pi_1,\pi_1\>D[\varphi] \> m_\x S(D[\varphi]) \epsilon \\
& = & \<\< \pi_0\pi_0,0,\>,0\>,\<\<0,\pi_0\pi_1\>,\pi_1\>\> D[\< \pi_0\pi_1 \varphi,\<\pi_0\pi_0,\pi_1\>D[\varphi] \> m_\x S(D[\varphi]) \epsilon \\
& = & \left\llbracket ((a,b),c)  \mapsto \diff{\epsilon(S(D[\varphi])(m_\x(\varphi(b),D[\varphi](c).x)))}{x}{0}{a}   \right\rrbracket  \\
& = & \left\llbracket ((a,b),c) \mapsto \epsilon(S(D[\varphi])(m_\x(\varphi b,\diff{D[\varphi](c)\cdot x}{x}{0}{a}))) \right\rrbracket \\
& = & \left\llbracket ((a,b),c) \mapsto \epsilon(S(D[\varphi])(m_\x(\varphi b,D[\varphi](c)\cdot a))) \right\rrbracket \\
& = &  \< \pi_0\pi_1 \varphi,\<\pi_0\pi_0,\pi_1\>D[\varphi]\>m_\x S(D[\varphi])\epsilon
\end{eqnarray*}

The key step uses the fact that $\eta(a)$ is the partial derivative of $\varphi(x)$ at $0$ with linear argument $a$, that is $\eta(a) = \diff{\varphi(x)}{x}{0}{a}$.  Otherwise we are using the fact that we are differentiating in linear positions so the term does not change.
\end{enumerate}

\begin{remark}{}
{\em Cartesian closed differential categories, following \cite{bem10,m12}, provide a semantics for the differential
and the resource $\lambda$-calculus.  Closed differential storage categories, by Section \ref{closedstorage},
are precisely those semantic settings which arise as the coKleisli categories of monoidal closed differential tensor storage categories.}
\end{remark}


\end{document}